\documentclass[11pt,leqno]{article}

\usepackage{amsthm}
\usepackage{amssymb,amsmath,amsfonts,microtype}
\usepackage{graphicx}
\usepackage{pdfsync}

\usepackage[margin=1in]{geometry}
\usepackage{amssymb}
\usepackage[hidelinks]{hyperref}
\usepackage[capitalize]{cleveref}
\usepackage{microtype}
\usepackage{esint}

\newtheorem{cor}{Corollary}[section]
\newtheorem{lem}{Lemma}[section]
\newtheorem{prop}{Proposition}[section]
\newtheorem{defn}{Definition}[section]

\newtheorem*{thm1}{Theorem 1}
\newtheorem*{thm1'}{Theorem 1'}

\newtheorem*{thm2'}{Theorem 2'}
\newtheorem*{thmB}{Theorem B}

\newtheorem*{thmA}{Theorem A}
\newtheorem*{thmC}{Theorem C}
\newtheorem*{thmD}{Theorem D}
\newtheorem*{thmE}{Theorem E}
\newtheorem*{thmB'}{Theorem B$^\prime$}
\newtheorem*{thmC'}{Theorem C$^\prime$}

\newtheorem*{propnA'}{Proposition A$^\prime$}
\newtheorem*{propnB'}{Proposition B$^\prime$}
\theoremstyle{remark}

\theoremstyle{definition}

\theoremstyle{remark}

\newcommand{\R}{\mathbb{R}}

\newcommand{\Z}{\mathbb{Z}}
\newcommand{\N}{\mathbb{N}}

\newcommand{\PP}{\mathbb{P}}
\newcommand{\EE}{\mathbb{E}}

\newcommand{\RR}{\mathcal{R}}

\newcommand{\FF}{\mathcal{F}}

\newcommand{\mP}{\mathcal{P}}

\newcommand{\mA}{\mathcal{A}}

\newcommand{\mF}{\mathcal{F}}

\newcommand{\QQ}{\underline{\mathcal{Q}}}
\newcommand{\bmR}{\overline{\mathcal{R}}}

\newcommand{\fW}{\mathfrak{W}}
\newcommand{\fw}{\mathfrak{w}}
\newcommand{\fh}{\mathfrak{h}}

\newcommand{\uu}{\underline{u}}
\newcommand{\vv}{\underline{v}}
\newcommand{\ux}{\underline{x}}
\newcommand{\uy}{\underline{y}}

\newcommand{\ub}{\underline{b}}
\newcommand{\uc}{\underline{c}}

\newcommand{\ue}{\underline{e}}
\newcommand{\uh}{\underline{h}}

\newcommand{\of}{\overrightarrow{f}}

\newcommand{\uom}{\underline{\om}}
\newcommand{\uOm}{\underline{\Om}}

\newcommand{\uQ}{\underline{Q}}
\newcommand{\uR}{\underline{R}}

\newcommand{\al}{\alpha}
\newcommand{\be}{\beta}
\newcommand{\la}{\lambda}

\newcommand{\La}{\Lambda}
\newcommand{\de}{\delta}

\newcommand{\De}{\Delta}
\newcommand{\ga}{\gamma}

\newcommand{\om}{\omega}
\newcommand{\Om}{\Omega}
\newcommand{\eps}{\varepsilon}
\newcommand{\si}{\sigma}

\newcommand{\1}{\mathbf{1}}

\newcommand{\subs}{\subseteq}
\newcommand{\ls}{\lesssim}

\newcommand{\prd}{\prod^{\ast}}

\newcommand{\nud}{\nu^{(d)}}
\newcommand{\zero}{\underline{0}}
\newcommand{\yyo}{\uy^{(0)}}
\newcommand{\yyv}{\uy^{(1)}}
\newcommand{\yomv}{y_1^{(\om_1)}}
\newcommand{\yoms}{y_s^{(\om_s)}}

\numberwithin{equation}{section}
\newcommand{\eq}{\begin{equation}
\newcommand{\ee}{\end{equation}}}

\title{Polynomial configurations in dense subsets\\ of the prime lattice}
\author{Andrew Lott, \'Akos Magyar, Giorgis Petridis and J\'anos Pintz}

\begin{document}

\maketitle

\begin{abstract} We provide a multidimensional extension of  previous results on the existence of polynomial progressions in dense subsets of the primes. Let $A$ be a subset of the prime lattice lattice $\PP^d$ of positive relative upper density. We show that $A$ contains all polynomial configurations of the form $\ux+P_0(y)\vv_0,\ldots,\ux+P_l(y)\vv_l$ satisfying a non-degeneracy condition, for some $\ux\in\Z^d$ and $y\in\N$. Moreover, if $A\subs \PP_N^d$ is of upper density $\de>0$ then one  may have that $0<y<\log^L N$ as long as $N$ is sufficiently large, where $L$ is a constant depending on the configuration but is independent of $N$ and $\de$. 

\footnote{The second author was supported by Grant {\#}854813 by the Simons Foundation. The third author was supported by Grant {\#  1723016} by the National Science Foundation. 
The fourth  author was supported by the Hungarian National Research Developments and Innovation Office, NKFIH, 147153 and KKP133819, Excellence 151341}


\end{abstract}

\section{Introduction} A celebrated result in analytic number theory, due to Green and Tao \cite{GT06}, states that the primes contain arbitrary long arithmetic progressions. In fact they have proved the following. Let $\PP_N$ denote the set of primes up to $N$.

\begin{thmA}\label{thmA} Let $\de>0$, $V=\{v_0,\ldots,v_l\}\subs\Z$. If $N\geq N(\de,V)$ and $A\subs \PP_N$ with $|A|\geq \de |\PP_N|$ then $A$ contains an affine image of $V$ i.e. a set $V'=\{x+yv_0,\ldots,x+yv_l\}$. 
\end{thmA}

It is easy to see that this implies that any subset of the primes of positive relative upper density contains an affine image of any finite set of integers. In \cite{Tao06} Tao has obtained an analogue of Theorem A for the Gaussian primes and asked if similar results hold for subsets of the prime lattice $\PP^d$. The first result in this direction was provided by B. Cook and the second author \cite{CM12} for finite sets $V\subs\Z^d$ which are in \emph{general position} in the sense that $|\pi_j(V)|=|V|$ for every $1\leq j\leq d$, where $\pi_j:\Z^d\to \Z$ denotes the orthogonal projection to the $j^{th}$-axis. 

\begin{thmB}\label{thmB} Let $d\geq 1,\ \de>0$ and let $V=\{\vv_0,\ldots,\vv_l\}\subs\Z^d$ be a set in general position. If $N\geq N(\de,V)$ and $A\subs \PP_N^d$ with $|A|\geq \de |\PP_N|^d$ then $A$ contains an affine image of $V$ i.e. a set $V'=\{\ux+y\vv_0,\ldots,\ux+y\vv_l\}$.
\end{thmB}

The full multidimensional extension of Theorem A were obtained independently in \cite{TZ15,CMT18} and \cite{FZ15} where it was shown that the conclusion of Theorem B holds for any finite set $V\subs \Z^d$. 

In \cite{TZ08} Tao and Ziegler has extended Theorem A in another direction, from arithmetic to polynomial progressions which are sets of the form $x+P_0(y),\ldots,x+P_l(y)$, where $\mP=(P_0,\ldots,P_l)$ is a family of integral polynomials. In \cite{TZ14,TZ18} they have provided quantitative bounds on the size of the parameter $y$ measuring the ``size" of the polynomial progression. 

\begin{thmC}\label{thmC} Let $\de>0$, $l\in\N$ and let $\mP=(P_0,\ldots,P_l)$, $P_i\in\Z[x], P_i(0)=0$. If $N\geq N(\de,l,\mP)$ and $A\subs \PP_N$ with $|A|\geq \de |\PP_N|$, then $A$ contains a polynomial progression $\{x+P_0(y),\ldots,x+P_l(y)\}$ for some $x,y\in\N$.

Moreover, one may have $0<y<\log^L N$, where $L=L(l,\mP)$ is a constant depends only the family of polynomials $\mP=(P_0,\ldots,P_l)$.
\end{thmC}

Note that the first statement in Theorem C was proved in \cite{TZ08} while the bound  $0<y<\log^L N$ with $L=L(l,P,\de)$ depending also on the density $\de$ was given in \cite{TZ14}. The dependence on $\de$ was removed in \cite{TZ18} by estimating polynomial averages of local Gowers norms by global ones via so-called concatenation theorems.\\

Corresponding lower bounds were also discussed in \cite{TZ14}, extending a construction in \cite{GPY09}, showing that $y\gg \log^l N$ is necessary, even for linear case, when $P_i(y)=iy$ for $0\leq i\leq l$.\\

Our aim in this note is twofold. On the one hand we show that Theorem C can be extended to the multi-dimensional setting for polynomial configurations that satisfy a similar non-degeneracy condition as given in Theorem B for affine linear configurations. On the other hand we'd like to present this result in an essentially self contained manner, relying only on a sieve theoretic estimate (see \eqref{2.6} below) and a quantitative version of a theorem of Bergelson-Leibman \cite{BL96} on polynomial configurations in dense subsets of the integer lattice. 

\begin{defn}\label{defn1.1} Let $l\geq 1$, $V=\{\vv_0,\ldots,\vv_l\}\subs\Z^d$ be a finite set and let $\mP=(P_1,\ldots,P_l)$ be a family of integral polynomials. We say that the polynomial configuration $\mP_V=(P_0(y)\vv_0,\ldots,P_l(y)\vv_l)$ is in general position, if for all $0\leq j< j'\leq l$ and $1\leq i\leq d$, 
\eq\label{1.1}
\deg\,\big(\pi_i(P_j(y)\,\vv_j - P_{j'}(y)\,\vv_{j'})\big) = 
\deg\,(P_j(y)\,\vv_j - P_{j'}(y)\,\vv_{j'}),
\ee
as a polynomial in $y$, $\pi_i$ being the natural projection to the $i^{th}$ coordinate axis.
\end{defn}

Our main result is the following. 

\begin{thm1}\label{thm1}
Let $\de>0$, $d,l\in\N$ and let $\mP_V=(P_0(y)\vv_0,\ldots,P_l(y)\vv_l)$ with $P_j\in\Z[y], P_j(0)=0$ be a polynomial configuration in general position. If $N\geq N(\de,\mP,V)$ and $A\subs \PP_N^d$ with $|A|\geq \de |\PP_N^d|$, then there exist $x,y \in\N$ with  $0<y<\log^L N$, such that
\eq\label{1.2}
\{\ux+P_0(y)\vv_0,\ldots,\ux+P_l(y)\vv_l\}\subs A,
\ee
where $L=L(\mP,V)$ is a constant depending only on initial data $\mP,V$.
\end{thm1}

\underline{Remarks.}

\begin{itemize}
 \item In dimension $d=1$, taking $v_0=\ldots =v_l=1$, Theorem 1 reduces to Theorem C.\\

\item If the set $V=\{\vv_0,\ldots,\vv_l\}$ is in general position then for any polynomial $P\in\Z[y]$, $P(0)=0$ the polynomial configuration $\mP_V:=(P(y)\vv_0,\ldots,P(y)\vv_l)$ is also in general position, hence Theorem 1 extends Theorem B.\\

\item If the polynomials $P_0,\ldots,P_l$ have distinct degrees then the pattern $\mP_V$ is in general position, as long as $\pi_i(\vv_j)\neq 0$ for all $1\leq j\leq l$ and $1\leq i\leq d$.\\


\item If $A=A_1^d$ where $A_1\subs \PP$ is of positive relative upper density, then Theorem 1 one follows from the 1-dimensional result of Tao-Ziegler \cite{TZ14} as the polynomial pattern $\mP_V(y)=\{P_0(y)\vv_0,\ldots,P_l(y)\vv_l\}$ is contained in the $d$-fold direct product of the polynomial progression $\mP(y):=\{\pi_i(P_j(y)\,\vv_j):\ 1\leq j\leq l,\ 1\leq i\leq d\}$.\\

\item In general, condition \eqref{1.1} is equivalent the following. If the leading term of $P_j(y)\,\vv_j - P_{j'}(y)\,\vv_{j'}$ is of the form $\uc_{jj'}\, y^{d_{jj'}}$ then $\pi_i(\uc_{jj'})\neq 0$ for all $1\leq i\leq d$ and $0\leq j<j'\leq l$.
\end{itemize}

\bigskip

Our approach follows that of \cite{TZ08,TZ14} and roughly consists of two parts. The starting point is to control polynomial averages of the form:
\eq\label{1.3}
\La_{\mP,V}(f_0,\ldots,f_l):=\EE_{x\in X}\EE_{y\in [M]} 
f_0(\ux+P_0(y)\vv_0)\ldots f_l(\ux+P_l(y)\vv_l),
\ee
where $X=(\Z/N\Z)^d$, $M:=\log^L N$, and we used the averaging notation $\EE_{a\in A} f(a):=\frac{1}{|A|} \sum_{a\in A} f(a)$.
The first part is to bound these averages in terms of polynomial averages of certain Gowers type local box norms. Recall that given an $s$-tuple of vectors $\uu_1,\ldots,\uu_s\in\Z^d$, we define the corresponding local box norm at scale $1\leq M\leq N$, as 
\eq\label{1.4}
\|f\|^{2^s}_{\Box_M(\uu_1,\ldots,\uu_s)}:= \EE_{\ux\in X}
\EE_{y_1^{(0)},\ldots,y_s^{(0)},y_1^{(1)},\ldots,y_s^{(1)}\in [M]} \prod_{\uom\in\{0,1\}^s} f(\ux+y_1^{(\om_1)}\uu_1+\ldots +y_s^{(\om_s)}\uu_s)
\ee
For bounded functions, when say $|f_i|\leq 1$, the so-called PET induction scheme of Bergelson-Leibman \cite{BL96} estimates the averages in \eqref{1.3} via an polynomial averages of box norms, namely one has
\eq\label{1.5}
|\La_{\mP,V}(f_0,\ldots,f_l)|\ls \min_{1\leq j\leq s}
\big(\EE_{\underline{h}\in [H]^t} \|f_j\|^{2^s}_{\Box_M(\uQ_1(\uh),\ldots,\uQ_s(\uh)}\big)^{\frac{1}{2^s}}+o(1),
\ee
where $\uQ_1,\ldots,\uQ_s:\Z^t\to\Z^d$ is a family of integral non-constant polynomial maps, and $H:=\log^{\sqrt{L}} N$ is a parameter, also referred to as the ``fine scale" \cite{TZ14}. By $o(1)$ we will denote a quantity that goes to $0$ as $N\to\infty$, that may depend on the initial configuration $\mP,V$ and parameters $d,l$.\\

In one dimension it was shown in \cite{TZ08} that such estimates remain true when the functions $f_i$ are not bounded uniformly in $N$, but bounded by a function $\nu_N$, referred to a pseudo-random measure \cite{GT06, TZ08}, that behaves very nicely with respect to polynomial type averages satisfying the so-called polynomial forms condition \cite[Definition 3.6]{TZ08}, described later. In our multidimensional setting it will be crucial to assume that the initial polynomial configuration $\mP_V$ is in general position in order to verify estimate \eqref{1.5} for a family of polynomials $\QQ=(\uQ_1,\ldots,\uQ_s)$ which also remain in general position.\\

The second part of the argument, is a transference argument of \cite{TZ14} adjusted to the multi-dimensional setting to obtain a decomposition of a function $f\geq 0$ bounded by a pseudo-random measure $\nud=\nud_N$ with respect to a given small quantity $\eps>0$. More precisely, to show that given a function $f:\Z^n\to\R$ satisfying $0\leq f\leq \nud$, one has the decomposition
\eq\label{1.6}
f=g+h,\quad0\leq g\leq 2\quad\textit{and}\quad \|h\|_{\Box_{H,M}(\uQ_1,\ldots,\uQ_s)}\leq\eps,
\ee
where we have used the notation,
\eq\label{1.7}
\|h\|_{\Box_{H,M}(\uQ_1,\ldots,\uQ_s)} := 
\big(\EE_{\uh\in [H]^t} \|f_j\|^{2^s}_{\Box_M(\uQ_1(\uh),\ldots,\uQ_s(\uh)}\big)
^{\frac{1}{2^s}}.
\ee
Note that both $\|\ \|_{\Box_M(\uu_1,\ldots,\uu_s)}$ and $\|\ \|_{\Box_{H,M}(\uQ_1,\ldots,\uQ_s)}$ are indeed norms.\\

The idea to use estimates like \eqref{1.5}-\eqref{1.6} to prove the existence of arithmetic progressions in the primes  originated and was crucial in \cite{GT06}, and later also in \cite{TZ08,TZ14} for polynomial progressions. In our case starting from a set $A\subs[N]^d$ of size $|A|\geq \de\, |\PP_N|^d$ one construct a function $f_A:[N]^d\to\R_{\geq 0}$ so that $\EE_{\ux\in [N]^d} f_A(\ux)\geq c_0\,\de$ (for some absolute constant $c_0>0$) and $0\leq f_A(\ux)\leq \log^{(d)}(\ux)$ with $\log^{(d)}(\ux)=\prod_{i=1}^d \log\,x_i$, simply by writing $f_A(\ux)=\1_A(\ux) \log^{(d)}(\ux)$. Then one may truncate the support of $f_A$ to $[\sqrt{N},N-\sqrt{N}]^d$ which changes its average only by a negligible amount, and then one may identifies $[N]^d$ with $X=(\Z/N\Z)^d$ in estimating the averages $\La_{\mP,V}(f_A,\ldots,f_A)$ as the total size of the polynomial configuration is $O(\log^{CL} N)$
which is much less that $\sqrt{N}$ thus there no``wraparound" issues.\\

If one can find a pseudo-random measure $\nu^{(d)}_N(\ux)\gg \log^{(d)}(\ux)$ for which \eqref{1.5} and \eqref{1.6} holds for functions $|f_i|\leq \nu^{(d)}_N$ and $0\leq f\leq \nu^{(d)}_N$, then by multi-linearity one has
\eq\label{1.8}
\La_{\mP,V}(f_A,\ldots,f_A)=\La_{\mP,V}(g,\ldots,g)+O(\eps).
\ee
Invoking the multi-dimensional polynomial extension of Szemer\'edi's theorem by Bergelson-Liebman \cite{BL96}, more precisely its more quantitative version given in \cite{TZ08} one has that 
\eq\label{1.9}
\La_{\mP,V}(g,\ldots,g)\geq c(\de),
\ee
for some constant $c(\de)>0$ for sufficiently large $N$, that may also depend on the initial data $\mP,V$. Choosing $\eps>0$ much smaller than $c(\de)$ it follows that
\eq\label{1.10}
\La_{\mP,V}(f_A,\ldots,f_A)>0,
\ee
which implies the existence of configurations $\ux+P_0(y),\ldots,\ux+P_l(y)$ in $A$ with $0<y<\log^L N$.\\

This cannot be done due to the irregularity of distribution of primes in small residue classes, however is possible after restricting $A$, and hence the support of $f_A$, onto a residue class $\ux\equiv\underline{b}\pmod{W}$, where $W=\prod_{p\leq w} p$ is the product of small primes $p\leq w$, with $w=w(N)$ is a function that is a sufficiently slowly growing to infinity with $N$, i.e. after applying the so-called $W$-trick \cite{GT06}.

\subsection*{Further directions.} If the polynomials $P_0,\ldots,P_l$ have distinct degrees it is expected that there are good explicit lower bounds for the constants $c(\de)$ in estimate \eqref{1.9}. In fact, in light of the recent quantitative breakthrough results the in one dimension by Peluse \cite{Pel20}, by Peluse-Prendiville-Shao \cite{PPS24},  Kravitz-Kuca-Leng \cite{KKL24a,KKL24b} and by Kosz-Mirek-Peluse-Wright \cite{KMPW24} it is plausible that one may take $c(\de) =\exp\exp\,(-\de^{-C})$ or equivalently $\de= (\log\log\,N)^{-c}$. The various $o(1)=o_{N\to\infty}(1)$ error terms in our arguments arise from the error term in Proposition \ref{prop2.1}, the so-called polynomial forms condition, are obtained via purely on number theoretic considerations are of $O((\log\log\log\log\,N)^{-1})$, see Section 2. Thus it is expected that similar poly-logarithmic bounds can be obtained on the density $\de$ of sets $A\subs\PP_N^d$.\\

We expect that our main result holds in full generality, without the assumption that the polynomial configuration $\mathcal{P}_V(y)$ being in general position, at least without the restriction $y \leq \log^L N$ on the "size" of the configuration. If the polynomials $P_1(y),\ldots,P_l(y)$
have the same degree $d$ but different main terms then such results may be obtained using the very recent quantitative result of Matthiesen-Ter\"{a}v\"{a}vainen-Wang \cite{MTW24} on the asymptotic distribution of polynomial progressions $\{x+P_1(y),\ldots,x+P_l(y)\}$ in $\PP_N$, combined with the sampling argument of Fox-Zhao \cite{FZ15}. Indeed, then one one may allow patterns whose projections to the coordinate axis forms such polynomial progressions, for example configurations of the form 
$\mathcal{P}_V(y)=\{P_{11}(y)\ue_1,\ldots,P_{l_1 1}(y)\ue_1,\ldots,P_{1d}(y)\ue_d,\ldots,P_{l_d d}(y)\ue_d\}$ where $\ue_1,\ldots,\ue_d$ being the standard basis vectors of $\R^d$. However, to extend this approach to polynomials whose degrees may not be the same, or to "small'' configurations one would need to prove some instances of the Bateman-Horn conjectures \cite{BH62} with are currently out of reach.

\section{The pseudo-random majorant and the polynomial forms condition.}

We follow \cite{TZ14} to define the the weighted indicator function of the set $A\subs \PP_{N'}$ truncated and restricted to an appropriate reduced residue class $\ux\equiv \ub \pmod{W}$. Let $w=w(N')=\frac{1}{10}\log\log\log\,N'$ and let $W=\prod_{p\leq w}$ ($p$ prime) and note that $W\ll (\log\log\,N')^{1/10}$ by the prime number theorem. We will set $N:=[N/W]$.\\ 

If $A\subs\PP_{N'}$ such that $|A|\geq \de |\PP_{N'}|$ then by pigeonhole principle, we can choose $\ub=(b_1,\ldots,b_d)$ so that $(b_i,W)=1$ for all $1\leq i\leq d$ and 
\eq\label{2.1}
|\{\ux\in[\sqrt{N},N-\sqrt{N}]^d,\ W\ux+\ub\in A\}| \geq 
\frac{\de}{2} \left(\frac{NW}{\phi(W)\log\,N}\right)^d,
\ee
for sufficiently large $N'$, using again the prime number theorem in arithmetic progressions. Define the function $f_A:[N]^d\to\R_{\geq 0}$,
\eq\label{2.2}
f_A (\ux) :=  \left(c_0\,\frac{\phi(W)\log\,N}{W}\right)^d\,\1_A(W\ux+\ub),
\ee
if $\ux\in [\sqrt{N},N-\sqrt{N}]^d$ and is equal to zero otherwise. Then by \eqref{2.1} we have that 
\eq\label{2.3}
\EE_{\ux\in [N]^d} f_A(\ux) = N^{-d}\sum_{\ux\in [N]^d} f_A(\ux) \geq \frac{c_0\de}{4},
\ee
where we used the averaging notation $\EE_{b\in B} f(b):=\frac{1}{|B|}\sum_{b\in B} f(b)$. Here $c_0>0$ is an absolute constant needed to be able to majorize the function $f_A(\ux)$ by the function $\nu^{(d)}(\ux)=\prod_{i=1}^d \nu_i(x_i)$, where is the $\nu$ pseudo-random measure used originally in \cite{TZ08,TZ14}, namely 
\eq\label{2.4}
\nu_{i}(x) := \frac{\phi(W)\log\,R}{W}\left(\sum_{d|Wx+b_i} \mu(d)\chi \left(\frac{\log\,d}{\log\,R}\right)\right),
\ee
for $x\in[N]$, where $R=N^{\eps_0}$, $\mu$ is the M\"{o}bius function and $\chi\in C^\infty(-1,1)$ is a smooth even function satisfying $\int_0^1 |\chi'(t)|^2 dt=1$ and $\chi(0)\geq 1/2$. It is easy to see that if we choose $c_0$ small enough, say $c_0=\frac{\eps_0}{10}$, then we have
\eq\label{2.5}
0\leq f_A(\ux)\leq \nu^{(d)}(\ux)\quad\textit{for all}\quad \ux\in X.
\ee
the pseudo-randomness property of the product measure $\nu^{(d)}$ we need is a multi-dimensional analogue of the so-called polynomial forms condition of Tao-Ziegler \cite[Proposition 3]{TZ14}, based on the following correlation estimate proved for the measure $\nu=\nu_b$ (uniformly for $b$ satisfying $(b,W)=1$). Let $h_1,\ldots,h_J$ be integers (not necessarily distinct) of size $h_j=O(\sqrt{N})$, then 
\eq\label{2.6}
\EE_{x\in [N]}\ \prod_{i=1}^J \nu(x+h_i) = 1+ o(1) +
O\bigg(\,Exp\,\bigg(\,O\,\bigg(\sum_{1\leq i<i'\leq I}\  
\sum_{\substack{w\leq p\leq R^{\log\,R}\\ p|h_i-h_{i'}}}\frac{1}{p}\bigg)\bigg)\bigg),
\ee
where $Exp (x)=e^x-1$, as long as $\eps_0$ is sufficiently small with respect to $J$.\\

\begin{defn} Let $\QQ=(\uQ_1,\ldots,\uQ_J):\Z^t\to\Z^d$ be a polynomial map. We say that the map $\QQ$ is non-degenerate if 
\eq\label{2.7}
\deg\,(\pi_i(\uQ_j-\uQ_{j'}))\geq 1,
\ee
for all $1\leq j<j'\leq J$, $1\leq i\leq d$, where $\pi_i$ is the natural projection the $i^{th}$ coordinate axis.
\end{defn}

Note in the special case $\QQ=\mP_V$ the notion of being non-degenerate is weaker then the notion of being in general position given in Definition \ref{defn1.1}. However this weaker notion is enough to derive the following key property of polynomial maps.

\begin{prop}\label{prop2.1} Let $t,d,D,J$ be fixed natural numbers, let $\eps_0$ sufficiently small and $L$ sufficiently large depending on $t,d,D,I$. Let $\QQ=(\uQ_1,\ldots,\uQ_J):\Z^t\to\Z^d$ be a non-degenerate integral polynomial map of degree at most $D$, with coefficients of size $O(W^C)$. Then
\eq\label{2.8}
\EE_{\uh\in [H]^t}\ \EE_{\ux\in X} \prod_{j=1}^J \nu^{(d)}(\ux+\uQ_j(\uh)) 
= 1+o(1),
\ee
where $H=\log^{\sqrt{L}} N$ and $X=(\Z/N\Z)^d\simeq [N]^d$.
\end{prop}

The following simple estimates involving the $Exp$ function will be useful in proving Proposition \ref{prop1}.

\begin{lem} Let $C_i>1$ and $\al_i>0$ for $1\leq i\leq d$. Then one has for sufficiently large $N$,

\eq\label{2.9} 
\prod_{j=1}^d (1+o(1)+C_i Exp\,(\al_i)\leq 1+o(1)+ \prod_{j=1}^d C_i\,Exp\,(\al_1+\ldots +\al_d)
\ee
\eq\label{2.10}
Exp\,(\al_1+\ldots +\al_d)\leq Exp\, (2^d\al_1)+\ldots +Exp\,(2^d\al_d).
\ee
\eq\label{2.11}
\prod_{j=1}^d (1+o(1)+C\,Exp\,(\al_i))\leq 1+o(1)+C^d \sum_{j=1}^d Exp\, (2^d\al_j)\quad (C>1).
\ee
\end{lem}

\begin{proof} For $d=2$, one has
\begin{align*}
&\ (1+o(1)+C_1 Exp\,(\al_1))\,(1+o(1)+C_2 Exp\,(\al_2))\\
&= 
1+o(1)+C_1(1+o(1))\,Exp\,(\al_1)+C_2(1+o(1))\,Exp\,(\al_2)+
C_1C_2\, Exp\,(\al_1)Exp\,(\al_2)\\
&= 1+o(1)+C_1C_2\,Exp\,(\al_1+\al_2)+C_1(1+o(1)-C_2)\,Exp\,(\al_1)
+C_2(1+o(1)-C_1)\,Exp\,(\al_2)\\
&\leq 1+o(1)+C_1C_2\, Exp\,(\al_1+\al_2),
\end{align*}
using $ Exp\,(\al_1)Exp\,(\al_2)=Exp\,(\al_1+\al_2)-Exp\,(\al_1)-Exp\,(\al_2)$. Then \eqref{2.9} follows by induction.

Writing $Exp\,(\al_i)=u_i$ for $i=1,2$, so $e^{\al_i}=1+u_i$, one has 
\[
Exp\,(\al_1+\al_2)=u_1+u_2+u_1u_2\leq 2u_1+2u_2+u_1^2+u_2^2=Exp\,(2\al_1)+Exp\,(2\al_2).
\]
Then \eqref{2.10} follows by induction on $d$ and \eqref{2.11} follows immediately from \eqref{2.9} and \eqref{2.10}.
\end{proof}

\begin{proof} [Proof of Proposition \ref{prop1}] We have by \eqref{2.6} and \eqref{2.11}
\begin{align*}
&\EE_{\ux\in X} \prod_{j=1}^J \nu^{(d)}(\ux+\uQ_j(\uh)) =
\prod_{i=1}^d \bigg(\EE_{x_i} \prod_{j=1}^J \nu_i(x_i+\pi_i(\uQ_j(\uh)))\,\bigg)\\
&\leq \prod_{i=1}^d \big(1+o(1)+C\,Exp\,(C\,\al_i(\uh)\big)
\leq 1+o(1) +C^d\,\sum_{i=1}^d Exp\,(C\,2^d\,\al_i(\uh)),
\end{align*}
where
\[\al_i(\uh)=\sum_{1\leq j<j'\leq I}\al_{j,j',i},\quad\textit{with}
\quad \al_{j,j',i}(\uh)=
\sum_{\substack{w\leq p\leq R^{\log\,R},\\\\ p|\pi_i(\uQ_j(\uh)-\uQ_{j'}(\uh))}}\,\frac{1}{p}.
\]

\medskip

Thus to prove \eqref{2.7} it is enough to show that for fixed $1\leq j<j'\leq I$ and $1\leq i\leq d$, one has

\eq\label{2.12}
\EE_{\uh\in [H]^t} \sum_{\substack{w\leq p\leq R^{\log\,R},\\\\ p|\pi_j(\uQ_i(\uh)-\uQ_{i'}(\uh))}}\,\frac{1}{p} = o(1).
\ee

\medskip

Note that by our assumptions the polynomial $\De Q(\uh):=\pi_i(\uQ_j(\uh)-\uQ_{j'}(\uh))$ has degree at least 1 and at most D.\\ 

To estimate the above sum first consider set $B$ of the primes $p\geq w$ for which the polynomial $\De Q$ identically vanishes$\pmod{p}$. Since all such primes $p$ must divide the coefficients of $\De Q$ which are of size $O(W^C)$ we have
\[w^{|B|}\leq \prod_{p\in B} p \leq W^C \ll e^{C\,w},\]
by the prime number theorem, thus 
\[\sum_{p\in B}\ \frac{1}{p}\, \leq\, \frac{|B|}{w}\, \ll \frac{1}{\log\,w}=o(1).\]

Next, we partition the sum in $p\notin B$ into dyadic intervals $2^k\leq p<2^{k+1}$, and estimate the contribution of each part depending the size of the parameter $k$.\\

Assume first that $2^{k/2}\leq \log\,H = \sqrt{L}\,\log\log\,N$. Since the polynomial $\De Q$ has degree at least 1 $\pmod{p}$, and $p$ is much smaller than $H$, the number of $\uh\in [H]^t$ satisfying $p|\De Q(\uh)$ is at most $H^t/p+O(1)$, thus 
\eq\label{2.13}
\EE_{\uh\in [H]^t} \sum_{2^k\leq p < 2^{k+1},\, p|\De Q(\uh)}\ 
\frac{1}{p}\ \leq\ \frac{1}{2^k} + O(H^{-1}).
\ee
Summing this for $2^k\geq w$ is at most $O(1/w)=o(1)$.\\

Assume now $2^{k/2} \geq \log H$. Then left side of \eqref{2.13} is estimated by,

\[\frac{1}{2^k}\ \EE_{\uh\in [H]^t}\,|\{2^k\leq p <2^{k+1};
\ p|\De Q(\uh)\}| \leq \frac{D\log\,H}{k\,2^k}+O(H^{-1})\ll \frac{1}{k2^{k/2}}+o(1).
\]

\medskip

Here we used the well-known fact that $|\{\uh\in [H]^t, \De Q(\uh)=0\}|=O(H^{t-1})$ and if $1\leq |\De Q(\uh)| \leq H^D$ then 
$\ \prod_{p|\De Q(\uh)} p\leq H^D\ $ thus the number of primes $p\geq 2^k$, $p|\De Q(\uh)$ is $O(D\log H/k)$.\\ 
This proves \eqref{2.12} and Proposition \ref{prop2.1} follows.
\end{proof}

Note that the $o(1)$ error term we obtained is $O(1/\log\,w)$ which is $O((\log\log\log\log\,N)^{-1})$ with the implicit constant may depending on the initial parameters $d,D,J$ but is independent of $t$.

\section{PET induction and the generalized von Neumann inequality.}

We describe the PET induction scheme originally devised by Bergelson and Leibman \cite{BL96} for bounded function and later extended by Tao and Ziegler to functions bounded by a pseudo-random measure $\nu$. We'll follow the notation of \cite{TZ08} except that in our multi-dimensional setting we will have $l$ shift operators in the directions of the vector $\vv_1,\ldots,\vv_l$, namely $T_j f(\ux):=f(\ux+\vv_j)$ for $1\leq j\leq l$.
We will estimate averages of the form 
\eq\label{3.1}
\La_{\mP}(\of) =\EE_{\ux\in X}\EE_{y\in [M]} \prod_{j=1}^d T_j^{P_j(y)} f_j(\ux),
\ee
for a family of functions $\of=(f_1,\ldots,f_l)$ satisfying $|f_j|\leq \nud$,
using repeated applications of van der Courput's lemma \cite{TZ08} which in our  context is the following simple observation. Let $x_n$ be a sequence of real numbers satisfying $x_n=O(\log^C N)$ uniformly in $n$. Let $M=\log^L N$ and $H=\log^{\sqrt{L}} N$ with $L$ sufficiently large with respect $C$. Then,
\eq\label{3.2}
\EE_{n\in [M]}\ x_n =\EE_{h\in H} \EE_{n\in [M]}\ x_{n+h}+o(1),
\ee
as shifting the average by a small amount $h$ will change it only slightly. Then by the Cauchy-Schwarz inequality,
\eq\label{3.3}
|\EE_{n\in [M]}\ x_n|^2 \leq \EE_{n\in [M]}\EE_{h,h'\in [H]}\  x_{n+h}\,x_{n+h'}+o(1),
\ee

After each application we will arrive at a new system of polynomials some of which will have reduced degrees in the $y$-variable or applied to the measure $\nud$ in place of a function $f_j$. Eventually, one arrives at a system where each polynomial corresponding to a shift applied to one of the functions $f_j$ is linear in $y$ but may depend on the number of additional small parameters $h_1,\ldots,h_t$. Averages over such systems are then majorized by polynomial averages of local box norms, described in the introduction.\\

\underline{Example}. Let us first illustrate the procedure with an example. Let $d=2$, $T_i f(\ux):= f(x+\vv_i)$ for $i=1,2$.
\[
\La(\of)=\EE_{\ux}\, \EE_y\ f_0\cdot T_1^y f_1\cdot T_2^{y^2} f_2,
\]
where $\ux\in X$, $y\in [M]$, and $|f_i|\leq \nu:=\nu^{(2)}$. We assume that $\vv_i$ has no zero coordinates so the system $\mP(y)=(\underline{0},y\vv_1,y^2\vv_2)$ is in general position. Then,
\[
|\La(\of)|\leq \EE_{\ux}\, \nu\, |\EE_y\, T_1^y f_1\cdot T_2^{y^2} f_2|,
\]
hence by van der Courput's lemma
\[
|\La(\of)|^2 \leq \EE_{\ux}\, \nu\, \EE_{h,h'} \EE_y
\,T_1^{y+h}f_1\cdot T_1^{y+h'}f_1\cdot T_2^{(y+h)^2}f_2\cdot T_2^{(y+h')^2}f_2 +o(1)
\]
Now we shift the variables $\ux$ to $\ux-y\vv_1$ which does not change the average but which amounts to multiplying each factor with $T_1^{-y}$, thus we get
\[
|\La(\of)|^2 \leq \EE_{\ux}\, \EE_{h,h'}\, T_1^h\nu\cdot T_1^{h'}\nu\ 
|\EE_y\, T_1^{-y}\nu\cdot T_1^{-y}T_2^{(y+h)^2}f_2\cdot T_1^{-y}T_2^{(y+h')^2}f_2| +o(1).
\]
By the polynomial forms condition $\EE_{\ux} \EE_{h,h'}\, T_1^h\nu\cdot T_1^{h'}\nu=1+o(1)$ hence after one more application of van der Courput's lemma,
\begin{align*}
|\La(\of)|^4 &\leq \EE_{\ux}\, \EE_{h,h'}\,\EE_{k,k'}\,T_1^h\nu\cdot T_1^{h'}\nu\ 
\ \EE_y\  T_1^{-y-k}\nu \cdot T_1^{-y-k'}\nu \cdot\\
&\cdot T_1^{-y-k}T_2^{(y+h+k)^2}f_2 \cdot T_1^{-y-k}T_2^{(y+h'+k)^2}f_2 
\cdot T_1^{-y-k'}T_2^{(y+h+k')^2}f_2\cdot T_1^{-y-k'}T_2^{(y+h'+k')^2}f_2
+o(1).\end{align*}
Note that there are four quadratic polynomials in $y$ however each have the same main term $y^2$ thus after shifting the variables $\ux$ to $\ux-y^2\vv_2$,  which amounts to multiplying each factor with $T_2^{-y^2}$, all exponents of the shift operators $T_1,T_2$ applied to the function $f_2$ will be linear. \\

Doing this procedure in general leads to polynomial averages of the form

\eq\label{3.4}
\La_{\RR}(\of):= \EE_{h_1,\ldots,h_t}\EE_{\ux}\ \EE_{y\in [M]}
\prod_{\al\in\mA} T^{\uR_\al(y,h_1,\ldots,h_t;W)} f_\al (\ux),
\ee
where 
\eq\label{3.5}
\uR_\al(y,h_1,\ldots,h_t;W) =\sum_{j=1}^l \,
R_{j,\al}(y,h_1,\ldots,h_t;W)\,\vv_j,
\ee
and
\begin{align}\label{3.6}
T^{\uR_{\al} (y,h_1,\ldots,h_t;W)} f_\al (\ux) 
&:= \prod_{j=1}^l T_j^{R_{j,\al}\,(y,h_1,\ldots,h_t;W)} f_\al (\ux)\\
&= f_\al (\ux + \uR_\al(y,h_1,\ldots,h_t;W)),\nonumber
\end{align}

using notation similar to \cite{BL96,TZ08}. For simplicity of notation we will wite $\uR_\al(\uh,y)$ suppressing the dependence on $W$. To describe the PET procedure in our settings we will need several definitions.\\ 


Let $\RR:=\{\uR_\al(y,\uh):\ \al\in \mA\}$ be a polynomial system with $\uR_\al\in\Z^d[y,h_1,\ldots,h_t]$ being an integral polynomial map for all $\al\in\mA$. We will consider $y$ as the primary variable of the polynomial maps $\uR_\al(y,\uh)$ and $\uh=(h_1,\ldots,h_t)$ as parameters. We say that the $y$-degree of a polynomial map $\uR(y,\uh)$ is equal to $d$ and write $\deg_y\,\uR(y,\uh)=d$, if 
\[
\uR(y,\uh)=\uc(\uh)\,y^d + \uQ(y,\uh),
\]
where $\uc(\uh)\neq \underline{0}$ and $\deg_y\,(\uQ(y,\uh))<d$, and $\uQ(y,\uh)=\underline{0}$ if $d=0$. We extend the notion of \emph{general position} to polynomial systems $\RR:=\{\uR_\al(y,\uh):\,\al\in \mA\}$ depending on parameters $\uh=(h_1,\ldots,h_t)$ as follows.

\begin{defn} We say that a polynomial system $\RR:=\{\uR_\al(y,\uh):\ \al\in \mA\}$ is in general position if for any two distinct nodes $\al,\be$, we have 
\eq\label{3.7}
\deg_y\,\big( \pi_i (\uR_\al-\uR_\be)\big) = 
\deg_y\,(\uR_\al-\uR_\be) \quad{for\ all}\ \ 1\leq i\leq d.
\ee
Moreover, if $\deg_y\,(\uR_\al-\uR_\be)=0$ then $\pi_i(\uR_\al)\neq \pi_i(\uR_\be)$.
\end{defn}

\medskip

For given nodes $\al,\be\in\mA$ let 
$d_{\al\be}:= \deg_y\,(\uR_\al(y,\uh)-\uR_\be(y,\uh))$
and write 
\eq\label{3.8}
\uR_{\al}(y,\uh)-\uR_{\be}(y,\uh) =\uc_{\al\be}(\uh)\,y^{d_{\al\be}} + \uQ_{\al\be}(y,\uh),
\ee
that is $\uc_{\al\be}(\uh)\neq\underline{0}$ is the coefficient of leading term of the polynomial $\uR_\al(y,\uh)-\uR_\be(y,\uh)$ in the $y$-variable. If $d_{\al\be}=0$, then we set $\uQ_{\al\be}=\underline{0}$. 
Then \eqref{3.7} is equivalent to
\eq\label{3.9}
\pi_i(\uc_{\al\be}(\uh))\neq 0\quad\textit{for all}
\quad 1\leq i\leq d,\ \ \al,\be\in\mA,\ \ \al\neq\be. 
\ee


An important note is that starting from an initial polynomial configuration $\mP(y)=(P_1(y)\vv_1,\ldots,P_l(y)\vv_l)$ in general position, all polynomial systems $\{\uR_\al(\uh,y):\,\al\in\mA\}$ will remain in general position during the PET procedure, i.e. they will satisfy \eqref{3.7} or equivalently \eqref{3.9}. To formalize this observation, write $\mA=\mA_0 \cup \mA_1$ where $\mA_0=\{\al:\ \deg_y\, \uR_\al(y,\uh)=0\}$, and define the \emph{doubling} of the system $\RR$ as the polynomial system:
\eq\label{3.10}
\RR^\prime =\{\uR_\be (\uh):\ \be\in\mA_0\}
\cup\{\uR_\be (y+h^1,\uh):\ \be\in\mA_1\}
\cup\{\uR_\be (y+h^2,\uh):\ \be\in\mA_1\}.
\ee
We index this system with the set of nodes $\mA':=\mA_0\cup\mA_1^1\cup\mA_1^2$, so that for nodes  $\be^1\in\mA_1^1$ and $\be^2\in\mA_1^2$ there correspond the polynomials
\eq\label{3.11}
\uR_{\be^1}(y,\uh,h^1) = \uR_\be (y+h^1,\uh),\ \ 
\uR_{\be^2}(y,\uh,h^2) = \uR_\be (y+h^2,\uh).
\ee

\medskip

\begin{lem}\label{lem3.1} Let $\RR=\{\uR_\al(\uh,y)):\,\al\in\mA\}$ be a polynomial system in general position. Then the its doubling, that is the system $\RR'$ defined in \eqref{3.10}-\eqref{3.11} is also in general position. 
\end{lem}

\begin{proof} We may assume without loss of generality that $\uR_{\al_0}=\underline{0}$ for some fixed node $\al_0$, as equation \eqref{3.7} is invariant under the translation $\uR_\al\to \uR_\al-\uR_{\al_0} \ \ (\al\in \mA)$. We write 

\[\uR_\al(y,\uh) =\uc_{\al}(\uh)\,y^{d_\al} + \uQ_{\al}(y,\uh),
\]
\\$\ $
where $\ d_\al = \deg_y\,\uR_\al(y,\uh)\,$ and $\ \uQ_{\al}(y,\uh)=0$ if $d_\al=\,0$, that is when $\al\in \mA_0$. By \eqref{3.9}, we have

\eq\label{3.12}
\pi_i(\uc_{\al}(\uh))\neq 0,\quad\textit{for all}\ \ 1\leq i\leq d.
\ee
\\$\ $
Let $\al,\be\in \mA\backslash{\{0\}}$ be two distinct nodes. If both nodes are in $\mA_0$, then $\uR'_\al=\uR_\al$ and $\uR'_\be=\uR_\be$ so \eqref{3.9} clearly holds for $\uR'_\al$ and $\uR'_\be$. If say $\al\in \mA_1^\si$ with $\si=1,2$, but $\be\in\mA_0$ then $d_\al>0$ and $d_\be=0$. Thus, 
\[
\uR'_{\al^\si}(y,\uh,h^\si)-\uR'_\be(y,\uh)=
\uR_\al(y+h^\si,\uh)-\uR_\be(\uh)=\uc_\al(\uh)y^{d_\al}+
\uQ'_{\al\be}(y,\uh,h^\si),
\]
\\$\ $
with $\deg_y\,(\uQ'_{\al\be}(y,\uh,h^\si))<d_\al$ and $\eqref{3.9}$ holds.\\

Assume now that $\al$ and $\be$ are both in $\mA_1$, and consider the difference 
\eq\label{3.13}
\uR'_{\al^\si}(y,\uh,h^\si)-\uR'_{\be^\tau}(y,\uh,h^\tau)=
\uR_\al(y+h^\si,\uh)-\uR_\be(y+h^\tau,\uh),
\ee
for $\si,\tau=1,2$. First we show that 
\[
\uR'_{\al^\si}(y,\uh,h^\si)-\uR'_{\be^\tau}(y,\uh,h^\tau)= \uc_{\al\be}(\uh)\,y^{d_{\al\be}} +
\uQ'_{\al^\si\be^\tau}(y,\uh),
\]
with $\deg_y\,(\uQ'_{\al^\si\be^\tau})(y,\uh)<d_{\al\be}$ in all cases, except when $d_\al=d_\be$, $\,\uc_\al(\uh)=\uc_\be(\uh)$ and $\si\neq\tau$.\\ 

Indeed, if $d_\al\neq d_\be$, say $d_\al>d_\be$ then the main term in the $y$-variable of the expression in \eqref{3.13} is $\uc_{\al}(\uh)\,y^{d_\al}=\uc_{\al\be}(\uh)\,y^{d_{\al\be}}$, thus \eqref{3.9} holds.\\

If $d_\al=d_\be$ and $\si=\tau$ then the expression in \eqref{3.13} takes the form

\begin{align*}
\uR_\al(y+h^\si,\uh)-\uR_\be(y+h^\si,\uh) &= \uc_{\al\be}(\uh)\,(y+h^\si)^{d_{\al\be}} +
\uQ_{\al\be}(y+h^\si,\uh)\\
&= \uc_{\al\be}(\uh)\,y^{d_{\al\be}} + \uQ'_{\al\be}(y,\uh,h^\si),
\end{align*}
\\$\ $
where $\deg_y\,(\uQ'_{\al\be}(y,\uh,h^\si))<d_{\al\be}$ and \eqref{3.9} holds again.\\

If $d_\al=d_be$,  $\si\neq\tau$ and $\uc_\al(\uh)\neq \uc_\be(\uh)$ then its is easy to see from \eqref{3.8} that $\uc_{\al\be}(\uh)=\uc_\al(\uh)-\uc_\be(\uh)$ and $d_{\al\be}=d_\al=d_\be$. In this case 
\begin{align*}
\uR_\al(y+h^\si,\uh)-\uR_\be(y+h^\tau,\uh) &= 
\uc_{\al}(\uh)\,(y+h^\si)^{d_\al} - \uc_{\be}(\uh)\,(y+h^\tau)^{d_\al} +
\uQ'_{\al\be}(y,\uh,h^1,h^2)\\
&= (\uc_{\al}(\uh)-\uc_{\be}(\uh))\,y^{d_\al} + \uQ''_{\al\be}(y,\uh,h^1,h^2),
\end{align*}
where $\deg_y\,(\uQ''_{\al\be}(y,\uh,h^1,h^2))<d_{\al}$ and \eqref{3.9} holds.\\

Finally, if $d_\al=d_\be$, $\,\uc_\al(\uh)=\uc_\be(\uh)\,$ but $\si\neq\tau$, then we use the expansion
\[\uR_\al (y,\uh)=\uc_\al(\uh)y^{d_\al} + \underline{e}_\al(\uh)y^{d_\al-1}+
\uQ_\al(y,\uh),
\]
where $\deg_y\,(\uQ_\al (y,\uh))<d_\al-1$ or $\uQ_\al(y,\uh)=0$ if $d_\al=1$. We have 

\begin{align*}
 &\ \uR_\al(y+h^\si,\uh)-\uR_\be(y+h^\tau,\uh) =\\
 &= \uc_{\al}(\uh)\,\big((y+h^\si)^{d_\al} - (y+h^\tau)^{d_\al}\big)+
\underline{e}_\al(\uh)(y+h^\si)^{d_\al-1}
-\underline{e}_\be(\uh)(y+h^\tau)^{d_\al-1}
+\uQ'_{\al\be}(y,\uh,h^1,h^2)\nonumber\\
&= \big(d_\al\,\uc_{\al}(\uh)(h^\si-h^\tau) + 
\underline{e}_\al(\uh)-\underline{e}_\be(\uh)\big)\,y^{d_\al-1} + 
\uQ''_{\al\be}(y,\uh,h^1,h^2),\nonumber
\end{align*}

where $\ \deg_y\,(\uQ_{\al\be}''(y,\uh))<d_\al-1\ $ or $\ \uQ''_{\al\be}(y,\uh)=0$ if $d_\al=1$. Applying the projection $\pi_i$ to the main term in the above expression, we have for all $1\leq i\leq d$,
\eq\label{3.14}
d_\al\,\pi_i(\uc_{\al}(\uh))\,(h^\si-h^\tau) + 
\pi_i(\underline{e}_\al(\uh)-\underline{e}_\be(\uh)) \neq 0.
\ee

Indeed, by our assumption $\pi_i(\uc_\al(\uh))\neq 0$, hence the first term in \eqref{3.14} depends on the new parameters $h^1$ and $h^2$ while the second term does not. Thus \eqref{3.9} holds again for the polynomials $\uR'_{\al^\si}(y,\uh,h^\si)$ and $\uR'_{\be^\tau}(y,\uh,h^\tau)$.\\
\end{proof}


We will fix a distinguished node $\al_0$ and we say that a node $\al$ is \emph{non-linear} if $\deg_y\, (\uR_{\al}-\uR_{\al_0})>1\,$ and \emph{linear} if $\deg_y\, (\uR_{\al}-\uR_{\al_0})=1\,$. We say that system $\RR$ is in \emph{general position w.r.t. $\al_0$}, if it is in general position and if for any two \emph{distinct linear nodes} $\al,\be$,
\eq\label{3.15}\deg_y\,(\uR_\al-\uR_\be) = 1.\ee

\begin{lem}\label{lem3.2} Let $\RR=$ be a polynomial system with distinguished node $\al_0$ and a nonlinear node $\al^\ast$. Then the its doubling, that is the system $\RR'$ defined in \eqref{3.10}-\eqref{3.11} is also in general position with respect to the distinguished node $\al_0^1$.
\end{lem}

\begin{proof} Without loss of generality we may assume $R_{\al^\ast}=0$ and hence $\deg R_{\al_0}\geq 2$. 
By Lemma \ref{lem3.1} it is enough to show that the main terms of the linear nodes with respect to the node $\al_0^1$ remain distinct in the new system $\RR'$. Let $\be,\ga$ be linear nodes w.r.t. $\al_0$. Write
\[\uR_\be(y,\uh)=\uR_{\al_0}(y,\uh)+{\uc}_\be(\uh)y+{\ue}_\be(\uh)\quad \textit{and similarly}\quad \uR_\ga(y,\uh)=\uR_{\al_0}(y,\uh)+\uc_\ga(\uh)y+\ue_\ga(\uh).
\]
Clearly, $\be,\ga$ both in $\mA_1$. If $\be\neq \ga$ then $\uc_\be\neq \uc_\ga$ thus clearly $\uR_\be(y+h^\si,\uh)-\uR_\ga(y+h^\si,\uh)$ is not constant in $y$ for $\si=1,2$. Now consider
\[
\uR_\be(y+h^1,\uh)-\uR_\ga(y+h^2,\uh) = 
\uR_{\al_0}(y+h^1,\uh)-\uR_{\al_0}(y+h^2,\uh)+(\uc_\be(\uh)-\uc_\ga(\uh))y+\ue_{\be\ga}(\uh).
\]
Taking a derivative w.r.t the $y$-variable, we have
\[
\partial_y (\uR_\be(y+h^1,\uh)-\uR_\ga(y+h^2,\uh)) =
\partial_y (\uR_{\al_0}(y+h^1,\uh)-\uR_{\al_0}(y+h^2,\uh))+(\uc_\be-\uc_\ga)(\uh)\neq \underline{0},
\]
as the above expression must depend on both $h^1$ and $h^2$ as $\al_0$ is a non-linear node. If $\be=\ga$ then we have again,
\[
\partial_y (\uR'_{\be^1}(y+h^1,\uh)-\uR'_{\be^2}(y+h^2,\uh))
=\partial_y (\uR_{\al_0}(y+h^1,\uh)-\uR_{\al_0}(y+h^2,\uh))\neq \underline{0}.
\]
\end{proof}


\medskip
 
Following \cite{BL96} we will assign a \emph{weight matrix} $\fW_\RR =\big(\fw_{j,k}\big)$ to a polynomial system 
$\,\RR= \{\uR_{\al}(\uh,y):\,\al\in\mA\}\,$ as follows. Let $\ \RR_j:=(R_{j,\al};\,\al\in\mA)\,$ and define $\,\bmR_j$ to be the set of those polynomials $\ R_{j,\al}(\uh,y)$ which satisfy
\eq\label{3.16}
\deg_y\,(R_{\al,j})\geq 1,\ \ but\ \ 
\deg_y\,(R_{j',\al})=0,\quad for\ all\ \ j<j'\leq l.
\ee
Let $D$ denote the maximum degree in $y$ of the polynomials $R_{j,\al}(y,\uh)$. For given $1\leq j\leq l$ and $1\leq k\leq D$ define the entry $\fw_{j,k}$ as the number of equivalence classes of the 
polynomials $R_{j,\al}\in\bmR_j$ of $y$-degree $k$, where two polynomials $R_{j,\al}$ and $R_{j,\al'}$ are equivalent if they have identical leading terms in the $y$-variable. The \emph{weight matrix} $\fW_\RR(\al^\ast) =\big(\fw_{j,k}(\al^\ast)\big)$ of the system $\ (\uR_{\al}:\,\al\in\mA)\,$ with respect to a specific node $\al^\ast$ is defined as the weight matrix of the shifted system 
$\ \{\uR_{\al}-\uR_{\al^\ast}:\,\al\in\mA\}\,$. Note that in the scalar case $l=1$ (and $v_1=1$) this agrees with definition of the weight vector $(w_k(\al^\ast))_{1\leq k\leq D}$ given in \cite{TZ08}.\\

We will estimate averages of the form \eqref{3.4} corresponding to systems $\ (\RR,\of):=\{(\uR_\al,f_\al):\,\al\in \mA\}$ where $|f_\al|\leq \nud $ for all $\al\in \mA$. 
We define a node $\al$ \emph{inactive} if $f_\al=\nud$ and \emph{active} otherwise. The \emph{weight matrix} $\fW_{\RR,\of}(\al^\ast)\,$ of the system $\,(\RR,\of)$ is defined as the weight matrix of the restricted polynomial system 
$(\uR_\al-\uR_{\al^\ast}:\,\al\in\mA_1)$, where $\mA_1$ denotes the set of active nodes.\\

We define the ordering of $l\times D$ weight matrices by reversed lexicographic ordering; we write $\fW'\prec \fW$ if there exist $1\leq j\leq l\leq l$, $1\leq k\leq D$ so that $\fw_{j,k}' < \fw_{j,k}$ but $\fw_{j',k'}'=\fw_{j',k'}'$ for $j<j'$ or $j=j'$ and $k<k'$. We remark that any ordered chain of weight matrices $\,\fW_1 \succ \fW_2\succ\ldots$ terminates in finitely many steps.\\


The key proposition of the PET procedure, which is a straightforward multidimensional extension of Proposition 5.14 in \cite{TZ08}, is the following.

\begin{prop}\label{prop3.1} Let $(\RR,\of)=\{(\uR_\al,f_\al):\,\al\in \mA\}$ be a polynomial system in general position with distinguished node $\al_0$. Assume the system $\RR$ is in general position with respect to $\al_0$, and there is an active non-linear node $\al^\ast$.\\
Then there exists a polynomial system 
$(\RR^\prime ,\of') =\{(\uR_{\al'}^{\prime},f'_{\al'}):\, \al'\in \mA'\}$ 
in general position with respect to a distinguished node $\al_0'$, and an active node $\al'$ such that,
\eq\label{3.17}
\fW_{\RR^\prime}(\al')\prec \fW_{\RR} (\al^\ast).
\ee
and 
\eq\label{3.18}
|\La_\RR(\of)|^2 \ll |\La_{\RR^{\prime}}(\of')| + o_{N\to\infty}(1),
\ee

\medskip

Moreover $f'_{\al_0}=f_{\al_0}$, and for each $\al'\in\mA'$ one has that $f'_{\al'}=f_\al$ for some $\al\in \mA$, or $f'_{\al'}=\nud$.
\end{prop}

\begin{proof} Shifting the system $\RR$ by $\uR_{\al^\ast}$, which amounts to changing the polynomials $\uR_\al\to\uR_\al-\uR_{\al^\ast}$, we may assume that $\uR_{\al^\ast}=0$, where $\al^\ast$ is a fixed non-linear node with respect to $\al_0$. Clearly the shifted system will remain in general position.\\ 

Write $\mA=\mA_0\cup\mA_1$, where $\mA_0=\{\be\in\mA:\ \deg_y\,(\uR_\be)=0\}$ and $\mA_1:=\mA\backslash \mA_0$. Accordingly, write
\[
G(x,\uh):=\prod_{\be\in\mA_0} T^{\uR_\be(\uh)} f_\be(\ux)
\]
and note that 
\[
|G(\ux,\uh)| \leq H(\ux,\uh):=\prod_{\be\in\mA_0} T^{\uR_\be(\uh)} \ \nud_\be(\ux).
\]

\medskip

Then by the polynomial forms condition we have that,

\eq\label{3.19}
\EE_{\uh\in[H]^t}\,\EE_{\ux\in X} H(\ux,\uh) = 1 +o(1).
\ee
Let 
\eq\label{3.20}
F(\ux,y,\uh):= \prod_{\be\in\mA_1} T^{\uR_\be(y,\uh)} f_\be(\ux),
\ee
then
\begin{align}\label{3.21}
\mid\La_\mP(\of)\mid 
&\leq\, \EE_{\uh\in[H]^t}\ \EE_{\ux\in X}\ H(\ux,\uh)\,
\mid\EE_{y\in [M]} F(\ux,y,\uh)\mid\\
\nonumber\\
&\leq \,\EE_{\uh\in[H]^t}\ \EE_{\ux\in X}\ H(\ux,\uh)\, \EE_{y\in [M]} 
\mid\EE_{h\in [H]}\,F(\ux,y+h,\uh)\mid + o(1).\nonumber
\end{align}

\medskip

Here the $o(1)$ term is coming from shifting the variables $y\to y+h$ causing an error term 
$O((\log\,N)^{C_\RR}\,H/M)=O((\log\,N)^{-\sqrt{L}/2}\,)=o(1)\,$ if $L$ is chosen sufficiently large with respect to $\RR$. Applying the Cauchy-Schwarz inequality in the $h$-variable, using \eqref{3.19}, we get

\begin{align}\label{3.22}
\mid\La_\mP(\of)\mid^2  
&\leq \,\EE_{(\uh,h^1,h^2)\in[H]^{t+2}}\EE_{\ux\in X}\,\EE_{y\in [M]}
\,H(\ux,\uh)\,F(\ux,y+h^1,\uh)\,F(\ux,y+h^2,\uh) + o(1)\\
\nonumber\\
&:= \La_{\RR^\prime}(\of')+o(1),\nonumber
\end{align}
where
\eq\label{3.23}
\RR^\prime =\{\uR_\be (\uh):\ \be\in\mA_0\}
\cup\{\uR_\be (y+h^1,\uh):\ \be\in\mA_1\}
\cup\{\uR_\be (y+h^2,\uh):\ \be\in\mA_1\}.
\ee
We index this system with the set of nodes $\mA':=\mA_0\cup\mA_1^1\cup\mA_1^2$, so that for the nodes  $\be^1\in\mA_1^1$ and $\be^2\in\mA_1^2$ we have the polynomials
\[
\uR_{\be^1}(y,\uh,h^1,h^2) = \uR_\be (y+h^1,\uh),\ \ 
\uR_{\be^2}(y,\uh,h^1,h^2) = \uR_\be (y+h^2,\uh)
\]
and functions
\[
f'_{\be^1}(\ux) = f'_{\be^2}(\ux) = f_\be (\ux).
\]

\medskip

Note that $\RR'$ is the doubling of the system $\RR$ and hence by Lemma \ref{lem3.1} and Lemma \ref{lem3.2} it is in general position with respect to the distinguished node $\al_0^1$. The doubling $R_{j,\al}(y,\uh)\to R_{j,\al}(y+h^\si,\uh)$ ($\si=1,2$) does not change the main terms of the polynomials $R_{j,\al}(y,\uh)$ in the $y$-variable, hence $\fW_{\RR'}=\fW_{\RR}$. In the above doubling procedure we do not activate inactive nodes but we may deactivate some active nodes, in fact those the active nodes $\al$ such that $\deg_y\,(\uR_\al)=0$. Thus $\fW_{\RR',\of'}\prec\fW_{\RR,\of}\,$ or $\,\fW_{\RR',\of'}=\fW_{\RR,\of}$.\\

Let $\fw_{j_0 k_0}$ be the first non-zero entry of the weight matrix $\fW_{\RR,\of}$, such an entry must exist as $\al_0$ is a nonlinear node. This means there is a polynomial $R_{j_0,\al}\in \RR_{j_0}$ such that $\deg_y\,(R_{j_0 \al})=k_0$ but $\bmR_j=\emptyset$ for $j<j_0$ and 
$\,\deg_y\,(R_{j_0 \be})\geq k_0\,$ for all 
$\,\bmR_{j_0,\be}\in \RR_{j_0}$.

Also $\,\deg_y\,(\RR_{j,\al})=0$ for $j>j_0\,$ by the definition of the reduced class of polynomials $\bmR_{j_0}$.\\

We claim that $\fw_{\RR',\of'}(\al^1)\prec \fw_{\RR',\of'}$ and hence $\fW_{\RR',\of'}(\al^1) \prec \fW_{\RR,\of}$, where $\fw_{\RR',\of'}(\al^1)$ is the weight matrix of the shifted system $\{\uR'_{\al'}-\uR'_{\al^1}:\ \al'\in \mA'\}$. Indeed if $j>j_0$ then $\,\deg_y\,(\RR'_{j,\al^1})=0$ thus we do not change the entries $\fw'_{j,k}$. If $k>k_0$ then again we do not change the entries $\fw'{j_0,k}$ as they depend only on the main terms of polynomials in $\RR'_{j_0}$ of degree $k$. However we reduce the entry $\fw'_{j_0,k_0}$ by 1 when we subtract the polynomial $R_{j_0,\al^1}(y,\uh)$ from the polynomials $R_{j_0,\al'}(y,\uh)\in \bmR_{j_0}$ as the equivalence class of $R_{j_0,\al^1}(y,\uh)$ vanishes. We may get new equivalence classes of smaller degrees than $k_0$ but that does not affect the ordering of the matrices. This proves \eqref{3.17} with and \eqref{3.18} follows from \eqref{3.22}.
\end{proof}


We have shown, after shifting the system so that $\uR_{\al_0}=\underline{0}$,
\begin{equation} \label{linearizedandweighted}
|\Lambda(f)|^{2^s}\leq 
\EE_{\underline{h}, y, \ux} f_{\al_0}(\underline{x})\prod_{\alpha\in A_{nl}}T^{\underline{R_{\alpha}}(\underline{h},y)}\nu^{(d)}(\underline{x})\times \prod_{\alpha\in A_{l}}T^{\ub_{\al} (\uh)y +\underline{c_{\alpha}}(\underline{h})}f_\al(\underline{x})+o(1),
\end{equation}

where $\ub_{\al} (\uh),\underline{c_{\alpha}}(\underline{h})$ are integral polynomial maps dependent on shift parameters $\underline{h}=(h_1,\ldots,h_t)$. We use $A_l$ to denote the active nodes ($l$ stands for linear), and we let $A_{nl}$ denote the inactive nodes. Note that polynomial system 
$\RR=\{\uR_\al(y,\uh),\al\in\mA_{nl}\}\bigcup 
\{\uR_\al(y,\uh)=\ub_{\al}(\uh)y+\uc_\al(\uh)\}$ is in general position in the sense of $\eqref{3.7}$. In particular, $\pi_i(\ub_{\al})\neq 0$ for all $i\in [d]$ and $\al\in \mA$.\\

We will show that the expression in \eqref{linearizedandweighted} is bounded by 
\begin{align}\label{gowersnormcontrol}
&\leq \bigg(\EE_{\uh}\ 
\EE_{\underline{x},\uy^{(0)}
,\uy^{(1)}}
\prod_{\omega\in\{0,1\}^{|A_l|}} T^{\sum_{\gamma\in A_l} y_{\gamma}^{(\omega_{\gamma})}\ub_{\ga} (\uh)} \,f_{\al_0} (\ux)\bigg)^{2^{-|A_l|}} + o(1)\\
&= \|f_{\al_0}\|_{\Box_{H,M}(\ub_{\ga}:\ \ga\in \mA_l)}
\nonumber
\end{align}
i.e. the average of the local Gowers box norms $\|f_{\al_0}\|_{\Box(\ub_{\ga}(\uh):\ \ga\in \mA_l)}$ +o(1).\\

The main tool in this step is the weighted generalized von Neumann inequality \cite[Appendix A]{TZ08}. To this end, we introduce new shift parameters $\uy=(y_{\alpha})_{\alpha\in A_l}$ and define $\uQ_0(\uy):=\sum_{\alpha\in A_l}\ub_{\alpha}(\uh)\,y_{\alpha}\,$. Then, after shifting $y$ by $-\sum_{\alpha\in \mA_l} y_{\alpha}$ and shifting the polynomial system by $\uQ_0(\uy)$, the right-hand side of (\ref{linearizedandweighted}) becomes 

\begin{align}
\EE_{\uh}\ \EE_{\uy,y,\ux}\,
T^{\underline{Q}_0 (\uy)}\, f_{\al_0}(\underline{x})&
\prod_{\alpha\in \mA_{nl}}
T^{\uQ_0(\uy)+\underline{R_{\alpha}}\left(\underline{h},y-\sum_{\gamma\in A_l}y_{\gamma}\right)}\nu^{(d)}\left(\underline{x}\right) \nonumber \\
&\times \prod_{\alpha\in A_{l}} T^{\ub_{\alpha}(\uh)y + \sum_{\gamma\in A_l} (\ub_\ga(\uh)-\ub_\al (\uh))
y_{\ga}+\uc_\al(\uh)} f_{\al} (\ux) +o(1). \label{shifted}
\end{align}

Notice that since we simply shifted the system, it is still in general position as a system in the variables $\uh,\uy$ and hence will satisfy \eqref{2.7} and the polynomial forms condition \eqref{2.8} applies.

Let 
\begin{align*}
f_{0,\underline{x},\underline{h},y}(\uy) &= T^{\uQ_0(\uy)}f_{\al_0}(\underline{x})\prod_{\alpha\in A_{nl}}T^{\uQ_0(\uy)+\underline{R_{\alpha}}\left(\underline{h},y-\sum_{\gamma\in A_l}y_{\gamma}\right)}\nu^{(d)}\left(\underline{x}\right),\\
f_{\alpha,\underline{x},\underline{h},y}(\uy) &= T^{\ub_{\al} (\uh)y+\sum_{\gamma\in A_l}(\ub_{\ga} (\uh)-\ub_{\al} (\uh))y_{\gamma}+\underline{c_{\alpha}}(\underline{h})}f_{\alpha}\left(\underline{x}\right),\\
\nu^{(d)}_{\alpha,\underline{x},\underline{h},y}(\uy) &= T^{\ub_{\al} (\uh)y+\sum_{\gamma\in A_l}(\ub_{\ga} (\uh)-\ub_{\al} (\uh))y_{\gamma}+\underline{c_{\alpha}}(\underline{h})}\nu^{(d)}\left(\underline{x}\right).
\end{align*}
With this in mind, (\ref{shifted}) becomes 
\[\EE_{\underline{h},y,\underline{x}}\,\EE_{\uy}\  f_{0,\underline{x},\underline{h}}(\uy) \prod_{\alpha\in A_l} f_{\alpha,\underline{x},\underline{h}}(\uy). 
\]

By design, $0\leq f_{\alpha,\underline{x},\underline{h},y}\leq \nu^{(d)}_{\alpha,\underline{x},\underline{h},y}\ $ pointwise and $f_{\alpha,\underline{x},\underline{h}}(\uy)$ is independent of $y_{\alpha}$ for every $\alpha\in A_l$. Thus, we may apply the generalized weighted von Neumann theorem in the $\uy$-variable \cite[Appendix A]{TZ08}: 
\begin{equation}
\EE_{\underline{h},y,\ux}\,\EE_{\uy}\  f_{0,\underline{x},\underline{h},y}(\uy)
\prod_{\al\in A_l} f_{\al,\ux,\uh}(\uy)\leq 
\EE_{\uh,y,\ux}\ ||f_{0,\ux,\uh,y} ||_{\square^{A_l}(\nu^{(d)})}\prod_{\al\in A_l} ||\nud_{\al,\ux,\uh,y}||_{\square^{A_l\setminus \{\alpha\}}}^{1\slash2}.
\end{equation}
Thus by Hölder's inequality, 
\eq\label{holder}
\EE_{\underline{h},y,\ux}\,\EE_{\uy}\  f_{0,\underline{x},\underline{h},y}(\uy)
\prod_{\al\in A_l} f_{\al,\ux,\uh}(\uy)\leq
 \bigg(\EE_{\uh,y,\ux}\ ||f_{0,\ux,\uh,y}||_{\square^{A_l}(\nu^{(d)})}^{2^{|A_l|}} \bigg)^{2^{-|A_l|}},  
\ee


as by the polynomial forms condition, we have for all $\al\in \mA_l$

\begin{align*}
\EE_{\uh,\ux}
||\nu^{(d)}_{\alpha,\underline{x},\underline{h},y,\cdot} ||_{\square^{A_l \setminus \{\alpha\}}}^{2^{|A_l|-1}}
&=\EE_{\underline{h},\ux,y}\ \EE_{\uy^{(0)},\uy^{(1)}}\  \prod_{\omega\in\{0,1\}^{A_l\setminus\{\alpha\}}}T^{\ub_{\al} (\uh)y+\sum_{\gamma\in A_l}(\ub_{\ga} (\uh)-\ub_{\al} (\uh)y_{\gamma}^{(\omega_{\gamma})}+\underline{c_{\alpha}}(\underline{h})}\nu^{(d)}\left(\underline{x}\right)\\
&=1+o(1).
\end{align*}


Note that,
\begin{equation}\label{weightedparavg}
\EE_{\underline{h}, y , \underline{x}}
||f_{0,\ux,\underline{h},y,\cdot}||_{\square^{A_l}(\nu^{(d)})}^{2^{|A_l|}}=
\EE_{\uh,\ux}\,\EE_{\uy^{(0)},\uy^{(1)}}\,
\prod_{\omega\in\{0,1\}^{A_l}} T^{\uQ_0(\uy^{(\omega)})}f(\underline{x})\ w(\underline{h},\underline{x},\uy^{(0)},\uy^{(1)}), 
\end{equation}
where
\begin{align}\label{weightfunct}
w(\underline{h},\underline{x},\uy^{(0)},\uy^{(1)})&:=
\EE_{y}\ \prod_{\alpha\in A_{nl}}\prod_{\omega\in\{0,1\}^{A_l}}T^{\uQ_0(\uy^{(\omega)})+\underline{R_{\alpha}}\left(\underline{h},y-\sum_{\gamma\in A_l}y_{\gamma}^{(\omega_{\gamma})}\right)}\nu^{(d)}\left(\underline{x}\right)\nonumber\\
&\times \prod_{\alpha\in A_l} \prod_{\omega^{(\alpha)}\in\{0,1\}^{A_l\setminus\{\alpha\}}}T^{\ub_{\al} (\uh)y+\sum_{\gamma\in A_l}(\ub_{\ga} (\uh)-\ub_{\al} (\uh))y_{\gamma}^{(\omega_{\gamma}^{(\alpha)})}+\underline{c_{\alpha}}(\underline{h})}\nu^{(d)}\left(\underline{x}\right).
\end{align}
Notice that the right-hand side of (\ref{weightedparavg}) is a weighted parallelogram average, and, ignoring the weight $w$, we end up with an average of local Gowers norms:
$$
\begin{aligned}
\EE_{\underline{h},\underline{x}}\,\EE_{\uy^{(0)},\uy^{(1)}}{\EE}\prod_{\omega\in\{0,1\}^{A_l}} T^{\uQ_0(\uy^{(\omega)})}f_{\al_0}(\underline{x})
&=\EE_{\underline{h},\underline{x}}\,\EE_{\uy^{(0)},\uy^{(1)}}\ \prod_{\omega\in\{0,1\}^{A_l}} T^{\sum_{\gamma\in A_l} y_{\gamma}^{(\omega_{\gamma})}\ub_{\ga} (\uh)}f_{\al_0}\left(\underline{x}\right)\\
&=\EE_{\uh}\ \|f_{\al_0}\|_{\Box(\ub_{\ga}(\uh):\ \ga\in \mA_l)}^{2^{|A_l|}}.\\\label{goal}
\end{aligned}
$$

Thus, it suffices to prove 
\begin{equation}
\underset{\uh,\ux,\uy^{(0)},\uy^{(1)}}{\EE}\prod_{\omega\in\{0,1\}^{A_l}} T^{\uQ_0(\uy^{(\omega)})}f_{\al_0}(\ux)\ (w(\uh,\ux,\uy^{(0)},\uy^{(1)})-1)= o(1) \label{done}
\end{equation}
By the Cauchy-Schwarz inequality and the bound $|f_{\al_0}|\leq \nu^{(d)}$, the square of the left-hand side of (\ref{done}) is at most 
$$
\begin{aligned}
&\textbf{} \underset{\uh,\ux,\uy^{(0)},\uy^{(1)}}{\EE}\prod_{\omega\in\{0,1\}^{A_l}} T^{\uQ_0(\uy^{(\omega)})}\nu^{(d)}\,(\ux)\big(w(\uh,\ux,\uy^{(0)},\uy^{(1)})-1\big)^2\\
&=\underset{\uh,\ux,\uy^{(0)},\uy^{(1)}}{\EE}\prod_{\omega\in\{0,1\}^{A_l}} T^{\uQ_0(\uy^{(\omega)})}\nu^{(d)}(\ux)\,
\big(w(\uh,\ux,\uy^{(0)},\uy^{(1)})^2
-2w(\uh,\ux,\uy^{(0)},\uy^{(1)})+1\big)\\
&=(1+o(1))-2(1+o(1))+(1+o(1)) =o(1),
\end{aligned}
$$
where we have repeatedly applied the polynomial forms condition to the coordinate projections of the polynomials appearing in \eqref{weightedparavg}-\eqref{weightfunct}, in the $\uh,\ux,\uy^{(0)},\uy^{(1)}$ and $y$ variables. To justify this, let $i\in [d]$ be fixed and consider the polynomials

\eq\label{poly1}
x_i+\pi_i\big(\uQ_0(\uy^{(\om)},\uh)\big)= \sum_{\al\in\mA_l}
\pi_i (\ub_\al(\uh))\,y_\al^{(\om_\al)},
\ee
for $\uom\in \{0,1\}^{\mA_l}$. The polynomials,
\eq\label{poly2}
x_i + \sum_{\al\in\mA_l}
\pi_i (\ub_\al(\uh))\,y_\al^{(\om_\al)}
+\pi_i\big(\underline{R_{\al'}}\,(\uh,y-\sum_{\ga\in A_l}y_{\ga}^{(\om_{\ga})})\big),
\ee
for $\al'\in\mA_{nl}$ and $\uom\in \{0,1\}^{\mA_l}$, 
and the polynomials
\eq\label{poly3}
x_i+ \pi_i(\ub_{\al}(\uh))\,y+
\sum_{\ga\in A_l}\pi_i\big((\ub_{\ga}-\ub_{\al})(\uh)\big)\, y_{\ga}^{(\om_{\ga}^{(\alpha)})} 
+ \pi_i(\underline{c_\al}(\uh)),
\ee

for $\al\in\mA_l$, $\uom^{(\al)}\in\{0,1\}^{A_l\setminus\{\al\}}$.\\

Since $\pi_i(\ub_\al (\uh))\neq 0$, the polynomials in \eqref{poly1} are distinct and satisfy \eqref{2.7}. The polynomials in \eqref{poly2} are distinct for the same reason, and they are also distinct from the polynomials in \eqref{poly1} as either $\deg_y\,\underline{R_{\al'}}\,(\uh,y)\geq 1$ hence they depend also on the $y$-variable, or $\pi_i(\underline{R_{\al'}})\neq 0$ hence contain a term depending only on the variables $\uh$. Finally the polynomials in \eqref{poly3} are distinct as 
$\pi_i\big((\ub_{\ga}-\ub_{\al})(\uh)\big)\neq 0$ for $\ga\neq \al$, and they are distinct from the polynomials 
in \eqref{poly1}-\eqref{poly2}, as they are independent of the $y_\al^{(0)},y_\al^{(1)}$ variables but depend on the $y$-variable.\\

This shows the validity of \eqref{gowersnormcontrol}. Since our initial family of functions is $f_1,\ldots,f_l$ we have that $\al_0=k$ and $f_{\al_0}=f_k$, and after re-indexing and renaming the variables, we may write 
\[
(\ub_\ga(\uh):\ \ga\in \mA_l) =
(\uQ^k_j(\uh^k):\ j\in [s_k])
\]
thus we have for all $1\leq k\leq l$,
\eq\label{gnormk}
\big| \La_{\mP,V}(\of)\big| \leq 
||f_k||^{2^{-S_k}}_{\Box_{H,M}(\uQ^k_j(\uh^k):\ j\in [s_k])}+o(1).
\ee
Note that the norms on right side of \eqref{gnormk} depend on $k$, however they can be majorized by a single local polynomial box norm using a simple concatenation property. Indeed one can define the concatenation of two polynomial systems $\QQ=(\uQ_1(\uh),\ldots,\uQ_s(\uh):\ \uh\in[H]^t)$ and $\QQ'=(\uQ_1'(\uh'),\ldots,\uQ_{s'}(\uh'):\ \uh'\in[H]^{t'})$ as $\QQ\oplus\QQ'=
(\uQ_1(\uh),\ldots,\uQ_s(\uh),\uQ_1'(\uh'),\ldots,\uQ_{s'}(\uh'):\ (h,h')\in [H]^{t+t'}$. Inductively one defines the concatenation $\QQ=\QQ^1\oplus\ldots\oplus\QQ^l$ of more that two polynomial systems $\QQ^1,\ldots,\QQ^l$. The following ``concatenation lemma" follows from simple application of 
H\"{o}lder's and the Cauchy-Schwarz inequality, see \cite[Lemma A.3]{TZ08}

\begin{lem}\label{lem3.3} Let $\QQ=\QQ^1\oplus\ldots\oplus\QQ^l$. Then one has for any $f:X\to\R$
\eq\label{concatlem}
||f||_{\Box_{H,M}(\QQ)}\geq ||f||_{\Box_{H,M}(\QQ^k)},
\ee
for all $1\leq k\leq l$.
\end{lem}

Note that if $\QQ^k$ is in general position and moreover if $\pi_i(\uQ^k_j)\neq 0$ for all $i\in[d],\ k\in [l]$ and $j\in [s_k]$ then $\QQ$ is also in general position. This holds in our case $\pi_i(\ub_\al)\neq 0$ for all $\al\in \mA_l$. Since $|f|_k\leq \nud$ and $\|\nud\|_{\Box_{H,M}(\QQ^k)}=1+o(1)$ by the polynomial forms condition as the system $\QQ^k$ is in general position. We have our key estimate by Lemma \ref{lem3.3} and inequality \eqref{gnormk}.

\begin{cor}\label{cor3.1}(\emph{generalized von-Neumann inequality})\\
Let $\mP=(P_1(y)\vv_1,\ldots,P_l(y)\vv_l)$ be a polynomial system in general position. Then there exist $t,s,S\geq 1$ and a polynomial system $\QQ=(\uQ_1,\ldots,\uQ_s)$ in general position, satisfying $\pi(\uQ_j)\neq 0$, such that
\eq\label{von-Neumann}
\big| \La_{\mP,V}(\of)\big| \leq \min_{1\leq k\leq l}\,
||f_k||^{2^{-S}}_{\Box_{H,M}(\uQ_1,\ldots,\uQ_s)}+o(1).
\ee
\end{cor}


\section{Dual functions and the main decomposition.}

In this section we prove the crucial decomposition \eqref{1.6} following \cite{TZ14} which combines Gowers approach to decompositions theorems based on the Hahn-Banach theorem \cite{Gow10}. The starting point is the show the orthogonality of $\nud -1$ to products of so-called \emph{dual functions} associated to polynomial averages of local box norms. 

\begin{defn}\label{defn4.1} Let 
$\QQ=(\uQ_1,\ldots,\uQ_s):\ [H]^t\to \Z^{sd}$ be a polynomial map where $\uQ_j\in\Z^d[h_1,\ldots,h_t]$. 
Let $f_{\uom}:X\to\R$, 
$\uom\in \{0,1\}^s\backslash \{\underline{0}\}$ 
be a family of functions. We define the dual function of the family of functions $(f_{\uom})_{\uom\in \{0,1\}^s\backslash \{\zero\}}$, as 
\eq\label{4.1}
D_{\QQ}(f_{\uom})(\ux):= 
\EE_{\uh\in [H]^t}\,\EE_{\yyo,\yyv\in [M]^s}\,
\prod_{\uom\in \{0,1\}^s,\uom\neq \zero} 
f_{\uom}\big(\ux+\yomv\uQ_1(\uh)+\ldots +\yoms\uQ_s(\uh)\big).
\ee
If $f_ {\uom} =f$ for all $\uom\in {\{0,1\}^s\backslash \{\zero\}}$ then we write $D_{\QQ}(f)$ and refer to $D_{\QQ}(f)$ as the dual function of the function $f$.
\end{defn}

This means that for fixed $\uh$ and fixed $\yyo=(y_1^{(0)},\ldots,y_s^{(1)})$, $\yyv=(y_1^{(1)},\ldots,y_s^{(1)})$
we take the product of the functions $f_{\uom}$ evaluated at the vertices $\ux+\yomv\uQ_1(\uh)+\ldots +\yoms\uQ_s(\uh)$ forming a parallelepiped as $\uom=(\om_1,\ldots,\om_s)$ runs through all ${\{0,1\}^s\backslash \{\zero\}}$, where the vectors $\uQ_1(\uh),\ldots,uQ_s(\uh)$ give the directions of the side vectors of the parallelepiped, and then we sum the products for all such parallelepipeds was $\yyo,\yyv$ are running through $[M]^s$ and $\uh$ through $[H]^t$. 
The terminology \emph{dual function} is originated in \cite{GT06} for global Gowers norms, and is due to the fact that the inner product,
\eq\label{4.2}
\langle f,D_{\QQ}(f)\rangle := 
\EE_{\ux\in X} f(\ux) D_{\QQ}(f)(\ux) = \EE_{\uh\in[H]^t}\,
\| f\|_{\Box_M (\uQ_1(\uh),\ldots,\uQ_s(\uh))}^{2^s}.
\ee


The Gowers type local box-norm inner product of a family of functions $(f_{\uom})_{\uom\in\{0,1\}^s}$ of scale $M$ and directions $\uu_1,\ldots,\uu_s$, is defined as
\eq\label{4.3}
\langle f_{\om}\rangle_{\Box_M(\uu_1,\ldots,\uu_s)}:=
\EE_{\ux\in X} \EE_{\yyo,\yyv\in [M]^s}\,
\prod_{\uom\in \{0,1\}^s} 
f_{\uom}\big(\ux+\yomv\uu_1+\ldots +\yoms\uu_s\big).
\ee

Let us recall a basic inequality referred to as the Gowers-Cauchy-Schwarz inequality for box norms that is well-known in various forms see e.g.\cite[Appendix B]{TZ08}. 

\begin{lem}\label{lem4.1} One has 
\eq\label{4.4}
\big|\langle f_{\om}\rangle_{\Box_M(\uu_1,\ldots,\uu_s)}\big| \leq 
\prod_{\uom\in\{0,1\}^s} \|f_{\uom}\|_{\Box_M (\uu_1,\ldots,\uu_s)}
\ee
\end{lem}

\begin{proof} Let us write.
\[
f_{\uom,\ux}(y_1,\ldots,y_s):=
f_{\uom}(\ux+y_1\uu_1+\ldots +y_s\uu_s).
\]
For a fixed $\ux\in X$, we show 
\eq\label{4.45}
\big|\EE_{\yyo,\yyv\in [M]^s}\,
\prod_{\uom\in \{0,1\}^s} 
f_{\uom,\ux}(\yomv,\ldots,\yoms)\big| \leq
\prod_{\uom\in \{0,1\}^s} \|f_{\uom,\ux}\|_{\Box_M},
\ee
where for a function $f:X\to\R$,
\[
\|f\|_{\Box_M}^{2^s} = \EE_{\yyo,\yyv\in [M]^s}
\prod_{\uom\in \{0,1\}^s} f(\yomv,\ldots,\yoms).
\]
This is easy to see by repeated application of the Cauchy-Schwarz inequality. Indeed, separating the last variables and writing $\yyo=({\yyo}',y_s^{(0)})$ $\yyv=({\yyv}',y_s^{(1)})$, one may write 
\begin{align*}
& \langle f_{\om}\rangle_{\Box_M(\uu_1,\ldots,\uu_s)} =\\
&\EE_{{\yyo}',{\yyv}'} \big(\EE_{y_s^{(0)}}\,
\prod_{\uom'\in\{0,1\}^{s-1}} 
f_{(\uom',0)}(y_1^{(\om_1)},\ldots,y_{s-1}^{(\om_{s-1})},y_s^{(0)})\big)\,
\big(\EE_{y_s^{(1)}}\,
\prod_{\uom'\in\{0,1\}^{s-1}} 
f_{(\uom',1)}(y_1^{(\om_1)},\ldots,y_{s-1}^{(\om_{s-1})},y_s^{(1)})\big)
\end{align*}
Then by the Cauchy-Schwarz inequality

\begin{align*}
& \langle f_{\om}\rangle_{\Box_M(\uu_1,\ldots,\uu_s)}^2
\leq \\
& \big(\EE_{{\yyo}',{\yyv}'}\,\EE_{y_s^{(0)},y_s^{(1)}}\,
\prod_{\uom\in\{0,1\}^s} 
f_{(\uom',0)}(y_1^{(\om_1)},\ldots,y_{s-1}^{(\om_{s-1})},y_s^{(0)})\,
f_{(\uom',0)}(y_1^{(\om_1)},\ldots,y_{s-1}^{(\om_{s-1})},y_s^{(1)})\big)\\
& \big( \EE_{{\yyo}',{\yyv}'}\,\EE_{y_s^{(0)},y_s^{(1)}}\,
\prod_{\uom\in\{0,1\}^s}
f_{(\uom',1)}(y_1^{(\om_1)},\ldots,y_{s-1}^{(\om_{s-1})},y_s^{(0)})\,
f_{(\uom',1)}(y_1^{(\om_1)},\ldots,y_{s-1}^{(\om_{s-1})},y_s^{(1)})\big)
\end{align*}
Notice that in each factors the functions $f_{\uom}$ depend only on the first $s-1$ components of $\uom$. Repeating the procedure in the other $s-1$ variables give \eqref{4.45}. Then by H\"{o}lder's inequality left side of \eqref{4.3} is estimated by
\begin{align*}
&\langle f_{\om}\rangle_{\Box_M(\uu_1,\ldots,\uu_s)}^{2^s}
= \bigg(\EE_{\ux}\,\EE_{\yyo,\yyv\in [M]^s}\,
\prod_{\uom\in \{0,1\}^s} 
f_{\uom,\ux}(\yomv,\ldots,\yoms)\bigg)^{2^s}\\
&\leq \big(\EE_{\ux}\,\prod_{\uom\in \{0,1\}^s} \|f_{\uom,\ux}\|_{\Box_M}\big)^{2^s}
\leq \prod_{\uom\in \{0,1\}^s} 
\big(\EE_{\ux}\,\|f_{\uom,\ux}\|_{\Box_M}\big)^{2^s}\big)
= \prod_{\uom\in \{0,1\}^s} 
\|f_{\om}\|_{\Box_M(\uu_1,\ldots,\uu_s)}^{2^s}.
\end{align*}
\end{proof}


The following orthogonality property is crucial for the decomposition \eqref{1.6}.

\begin{prop}\label{prop4.1} Let $K,s\geq 1$ and for $1\leq k\leq K$ let $f^k_{\uom}:X\to [-1,1]$ be given for $\uom\in \{0,1\}^s\backslash \{\zero\}$. Let $\QQ=(\uQ_0,\uQ_1,\ldots,\uQ_s)$, $\uQ_0=\zero$, $\uQ_j\in \Z^d[\uh]$ be a family of polynomial maps in general position, satisfying condition \eqref{2.7} for all $1\leq i\leq d$ and $1\leq j<j'\leq s$. Then one has 
\eq\label{4.5}
\langle (\nud-1),\prod_{k=1}^K D_{\QQ}(f^k_{\uom})\rangle :=\EE_{\ux\in X}\ \big(\nud(\ux)-1\big)\cdot \prod_{k=1}^K D_{\QQ}(f^k_{\uom})(\ux) = o(1),
\ee
where the implicit constant may depend on $K,s$ and the polynomial map $\QQ$, 
\end{prop}

\begin{proof}  By expanding the product of the dual functions, we have 

\eq\label{4.6}
\prod_{k=1}^K D_{\QQ}(f^k_{\uom})\,(\ux) =
\EE_{\uh_1,\ldots,\uh_k}\,
\EE_{\substack{\yyo_1,\ldots, \yyo_k\\ \yyv_1,\ldots,\yyv_K}}\
\prod_{k=1}^K \prd_{\uom\in \{0,1\}^s} 
f^k_{\uom}\,\big(\ux+y_{k1}^{(\om_1)}\uQ_1(\uh_k)+\ldots +y_{ks}^{(\om_s)}\uQ_s(\uh_k)\big),
\ee

where we have written $\prd_{\om}$ for the restricted product $\prod_{\uom\in \{0,1\}^s,\uom\neq 
\zero}$ for simplicity of notations.\\

Let $B_{K,t}(H)$ be the set of those $k$-tuples of vectors
$(\uh_1,\ldots,\uh_K) \in [H]^{kt}$ such that
$\pi_i(\uQ_j(\uh_k))=0$ for some $i\in [d],\,j\in [s]$ and $k\in [K]K$. Then $|B_{K,t}(H)|\ll H^{kt-1}$ by the Schwarz-Zippel lemma, see \cite[Lemma D.3]{TZ08}, since $\pi_i(\uQ_j(\uh))$ is not zero by our assumption. Thus the contribution of the $K$-tuples that are in $B_{k,t}(H)$ to the right side of \eqref{4.4} is $O(H^{-1})=o(1)$.\\

Let $(\uh_1\,\ldots,\uh_K)\notin B_{k,t}(H)$ be a fixed $K$-tuple. We estimate 
\eq\label{4.7}
\EE_{\ux}\big(\nud(\ux)-1\big)\ 
\EE_{\substack{\yyo_1,\ldots, \yyo_k\\ \yyv_1,\ldots,\yyv_K}}\
\prod_{k=1}^K \prod^\ast_{\uom\in \{0,1\}^s} 
f^k_{\uom}\,\big(\ux+y_{k1}^{(\om_1)}\uQ_1(\uh_k)+\ldots +y_{ks}^{(\om_s)}\uQ_s(\uh_k)\big).
\ee

\medskip

Note that the product is taken over the union of $K$-distinct parallelepipeds with vertices $\ux+y_{k1}^{(\om_1)}\uQ_1(\uh_k)+\ldots +y_{ks}^{(\om_s)}\uQ_s(\uh_k)$. We may embed these $K$ parallelepipeds into a $Ks$-dimensional parallelepiped with vertices of the form
\[
\ux+ \sum_{j=1}^s \sum_{k=1}^K\ 
y_{k1}^{(\om_{k1})}\uQ_1(\uh_k)+\ldots +y_{ks}^{(\om_{ks})}\uQ_s(\uh_k),
\]
corresponding to 
$\underline{\Om}=(\uom_1,\ldots,\uom_K)$, $\uom_k=(\om_{k1},\ldots,\om_{ks})\in\{0,1\}^s$. 
Define $f_{\uOm}=1$ unless 
$\Om=(\zero,\ldots,\zero,\uom,\zero,\ldots,\zero)$, i.e. when $\uom_j=\zero$ for all $j\neq k$ and $\uom_k=\uom$; in which case let $f_{\uOm}=f^k_{\uom}$. In other words, to all the remaining vertices of $Ks$-dimensional parallelepiped we attach weight 1. Then the expression in \eqref{4.7} may be written as 
\begin{align}\label{4.8}
&\EE_{\ux}\,\big(\nud(\ux)-1\big)\ 
\EE_{\yyo_1,\ldots, \yyo_K, \yyv_1,\ldots,\yyv_K}\ 
\prod_{\uOm\in\{0,1\}^{Ks}}
f_{\uOm}\,
\big(\ux +\sum_{k=1}^K\sum_{j=1}^s 
y_{kj}^{(\om_{kj})} \uQ_j(\uh_k)\big)
\\
&= \big\langle \nud -1, f_{\uOm} 
\big\rangle_{\Box_M (\uQ_j(\uh_k):\,j\in [s], k\in [K])}.\nonumber
\end{align}

\medskip

By the Gowers-Cauchy-Schwarz inequality for box norms we have that 
\eq\label{4.9}
\big|\big\langle \nud -1, f_{\uOm} 
\big\rangle_{\Box_M (\uQ_j(\uh_k):\,j\in [s], k\in [K])}
\big| \leq 
\|\nud -1\|_{\Box_M (\uQ_j(\uh_k):\,j\in [s], k\in [K])},
\ee

as $|f_\Om|\leq 1$ and hence $\,\|f_{\uOm}\|_{\Box_M (\uQ_j(\uh_k):\,j\in [s], k\in [K])}\leq 1\,$ for all $\Om\in \{0,1\}^{Ks}$. Thus by H\"{o}lder's inequality the left side of \eqref{4.6} is estimated by 
\begin{align*}
&\EE_{\uh_1,\ldots,\uh_k}\|\nud -1\|_
{\Box_M (\uQ_j(\uh_k):\,j\in [s], k\in [K])}
\leq 
\bigg(\EE_{\uh_1,\ldots,\uh_K}\|\nud -1\|^{2^{Ks}}_{\Box_M (\uQ_j(\uh_k):\,j\in [s], k\in [K])}\bigg)^{\frac{1}{2^{Ks}}}
=o(1),
\end{align*}
\medskip

by the polynomial forms condition \eqref{2.8} applied in the variables $\ux_i,y_{kj}^{(\si)},\uh_{k}$ $(i\in [d],j\in [s],\,k\in [K])$ to the family of polynomials
$\ x_i +\sum_{k=1}^K\sum_{j=1}^s\,
y_{kj}^{(\om_{kj})} \pi_i\big(\uQ_j(\uh_k)\big)$.
\end{proof}

\medskip

Using the orthogonality property \eqref{4.5} one can obtain the crucial decomposition \eqref{1.6} of an unbounded function $0\leq f\leq \nud$ via an abstract decomposition theorem which appears implicitly in \cite{GT06} and is stated in various essentially explicit forms in \cite{RTTV08,Gow10,TZ14}. We recall the statement from \cite{TZ14} and for the sake of completeness provide a proof, which is a slight modification of an argument of Gowers \cite[Theorem 4.8]{Gow10}. 


\begin{thmD}{(Dense model theorem \cite{TZ14})}\label{thmD}

Let $\eps>0$ and let $\mF$ be a set of bounded functions $F:X\to [-1,1]$. Then there exists a constant $C>0$ and a constant $\eta>0$ depending only on $\eps$ such that the following holds; If a function $\nu:X\to\R_{\geq 0}$ obeys the bound 
\eq\label{4.10}
\EE_{\ux\in X} 
\big(\nu(\ux)-1\big)\,F_1\cdot\ldots\cdot F_K(x)\leq\eta,
\ee
\medskip
for all $K\leq \eps^{-C}$ and $F_j\in\mF$, then to every function $0\leq f\leq \nu$ there exists a function $0\leq g\leq 2$ such that 
\eq\label{4.11}
\big| \EE_{\ux\in X} 
\big( f(\ux)-g(\ux)\big)\,F(\ux)\,\big| \leq \eps,
\ee
for all $F\in\mF$.
\end{thmD}

This means that if the function $\nu-1$ is approximately  orthogonal to long products of functions $F_i\in \mF$ then every function $0\leq f\leq \nu$ can be approximated by a bounded function $g$ so that $f-g$ is approximately orthogonal to all function $F\in\mF$. If $1 \in\mF$ (as in most cases of interest) then $f$ and $g$ has approximately the same average, and $g$ is referred to as a dense model of $f$ for the family of functions $\mF$. 

\begin{proof}

Let $\epsilon>0$, and define $J: \R\to \R$ by $J(x)=\frac{x+|x|}{2}$. By the Weirerstrass approximation theorem, there exists a polynomial $P(x)=\sum_{i=0}^j a_i x^i$ such that 
\begin{equation}|P(x)-J(x)|<\frac{1}{8}\label{approximation} \end{equation}
 for all $x\in [-2\eps^{-1}, 2\eps^{-1}]$. Let $C$ be large enough that $\eps^{-C}>j$, and let $\eta=\frac{1}{4}\left(\sum_{i=0}^j |a_i| (2\eps^{-1})^i\right)^{-1}$.

Let $\nu:X\to\R_{\geq 0}$ be a function satisfying \eqref{4.10}, and suppose for the sake of contradiction there exists a function $0\leq f\leq \nu$ which cannot be written as $g+h$ where $0\leq g\leq 2$ and 
\[\sup_{F\in \mF} \big| \EE_{\ux\in X} h(\ux)F(\ux)\big| \leq \eps.\]
By the Hahn-Banach theorem (specifically \cite[Corollary 3.2]{Gow10} ), there exists $\phi:X\to \R$ such that $\langle f, \phi \rangle >1$, $\langle g, \phi\rangle < 1\slash 2$
for every $0\leq g \leq 1$, and $\langle h,\phi \rangle \leq \eps^{-1}$ for every $h: X\to \R$ such that 
\[\sup_{F\in \mF} \big| \EE_{\ux\in X} h(\ux)F(\ux)\big| \leq 1.\]
The above condition forces $\phi$ to be in the subspace spanned by the functions $F\in \FF$, and then by\\ \cite[Corollary 3.5]{Gow10}, we may write $\phi = \sum_{i=1}^s \la_i F_i$ 
where $F_1,\ldots,F_s\in \mF$ and \begin{equation}\sum_{i=1}^s |\la_i|<2\eps^{-1}\label{phi decomp}.\end{equation} 

By \eqref{4.10} and some manipulation, we have 

\begin{align}
\big|\langle \nu-1, P\phi \rangle\big| &\leq \sum_{i=0}^j |a_i|\big|\big\langle \nu-1, \bigg(\sum_{k=1}^s \la_k F_k\bigg)^i\big\rangle\big| \\
&\leq \eta \sum_{i=0}^j |a_i|\left(\sum_{k=1}^s |\lambda_k|\right)^i \nonumber\\
&\leq \eta \sum_{i=0}^j |a_i|\left(2\eps^{-1}\right)^i\nonumber\\ 
&= 1/4. \label{pphi bound}\nonumber
\end{align}

Since $\langle g, \phi \rangle\leq 1\slash 2$ for every $0\leq g\leq 1$, we have $\langle 1, J\phi \rangle\leq 1\slash 2$. By \eqref{phi decomp}, it is clear that $|\phi|<2\eps^{-1}$ pointwise, and so by \eqref{approximation}, 
 $||P\phi -J\phi||_{\infty}\leq 1\slash 8$. Combining these facts, we have $\langle 1,P\phi \rangle \leq 5\slash 8$. Hence, by \eqref{pphi bound}, $\langle \nu, P\phi \rangle \leq 7\slash 8$. Finally, once again using the fact that 
$||P\phi -J\phi||_{\infty}\leq 1\slash 8$, we have $\langle \nu, J\phi \rangle \leq 1$. But then 
\[\langle f, \phi \rangle \leq \langle f, J\phi \rangle \leq \langle \nu, J\phi \rangle \leq 1,\]
contradicting the fact that $\langle f, \phi \rangle >1$.

\end{proof}

Applying Proposition 4.1 and Theorem D to the family $\mF :=\{D_{\QQ}(f_{\uom})_{ \uom\in\{0,1\}^s\backslash\{\zero\}},\,f_{\uom}:X\to [-1,1]\}$ where $\QQ=(\uQ_0,\uQ_1,\ldots,\uQ_s)$ is a family of integral polynomials, we have 

\begin{cor}  Let $s,t\geq 1$ and let $\QQ=(\uQ_0,\uQ_1,\ldots,\uQ_s)$, $\uQ_0=\zero$, $\uQ_j\in \Z^d[h_1,\ldots,h_t]$ be a family of polynomials in general position. Given $\eps>0$, if $\,N\geq N(\QQ,t,s,\eps)\,$ are sufficiently large then to every function $0\leq f\leq \nud$ there exists a function $0\leq g\leq 2$ s.t.
\eq\label{4.12}
\big| \EE_{\ux\in X} \big( f(x)-g(\ux)\big) D_{\QQ} (f_{\uom})(\ux) \big| \leq \eps, 
\ee

\medskip

for all families of functions 
$\big(f_{\uom}\big)_{ \uom\in\{0,1\}^s\backslash\{\zero\}}$, $f_{\uom}:X\to [-1,1]$.
\end{cor}


At this point, for the function $0\leq f_A\leq \nud$ we may obtain a decomposition $f_A =g+\fh$ with $0\leq g\leq 2$ and the function $\fh$ satisfying
\eq\label{4.13}
\big| \EE_{\ux\in X} \fh(\ux) D_{\QQ} (f_{\uom})(\ux)\big|\leq\eps,
\ee
for all dual functions $D_{\QQ} (f_{\uom})$ corresponding to bounded functions $f_{\uom}:X\to [-1,1]$. Note that $\fh$ may not be bounded independently of $N$ but $|\fh|\leq \nud+2$. If \eqref{4.11} would hold for $D_{\QQ}(\fh)$ i.e. when $f_{\uom}=\fh$ for all $\uom\in\{0,1\}\backslash \{\zero\}$, then we would have 
\eq\label{4.12}
\|\fh\|_{\Box_{H,M}(\uQ_1,\ldots,\uQ_s)} \leq \eps.
\ee
Hence by \eqref{1.5} and multi-linearity
\eq\label{4.15}
\La_{\mP,V}(f_A,\ldots,f_A)=\La_{\mP,V}( g,\ldots,g)
+O(\eps^c)+o(1),
\ee
where $c>0$ is a constant depending on the initial data $\mP$ and $V$. The Bergelson-Liebman theorem \cite{BL96} implies (see next section), 
\eq\label{4.16}
\La_{\mP,V}( g,\ldots,g)\geq c(\de)+o(1),
\ee
where $c(\de)>0$ is a constant depending on the initial data but is independent of $N$. Thus choosing $\eps>0$ sufficiently small we have that 
$\La_{\mP,V}(f_A,\ldots,f_A)>0$ which implies our main result.\\

The idea that the bounded functions $f_{\uom}$ can be replaced by the unbounded function $\fh$ appeared first in \cite{CFZ15}. The proof is essentially the same as that of \cite[Theorem 11]{TZ14} and include it here only for the sake of completeness. Under the same conditions as in Proposition \ref{prop4.1}, we have 

\begin{prop}\label{prop4.2} Let $\eps>0$ and let $\fh:X\to\R$ s.t. $|\fh|\leq \nud+2$. If 
\eq\label{4.17}
\big|\EE_{\ux\in X}\ \fh(\ux)\,D_{\QQ}(f_{\uom})(\ux)\big| \leq \eps+o(1),
\ee
for all bounded families of functions $(f_{\uom})_{\uom\in\{0,1\}^s\backslash{\{\zero\}}}$, $f_{\uom}:X\to [-1,1]$, then 
\eq\label{4.18}
\|\fh\|_{\Box_{H,M}(\uQ_1,\ldots,\uQ_s)} \leq \eps^C +o(1)
\ee
for some constant $C>0$ that is independent of $\eps$.
\end{prop}

\begin{proof} By multi-linearity it is enough to show that 
\eq\label{4.20}
\langle f_{\uom}\rangle_{\Box_{H,M}(\uQ_1,\ldots,\uQ_s)}\ll
\eps^C + o(1)
\ee
whenever $|f_{\uom}|\leq 1$ or $|f_{\uom}|\leq\nud$, assuming that at least one of the functions $f_{\uom}=\fh$. We prove \eqref{4.20} on the number $n$ unbounded functions $f_{\uom}$.\\

If $n=1$ we may assume $f_{\zero} =\fh$ and \eqref{4.20} follows from \eqref{4.16}. If $n\geq 2$ we may assume again that $|f_{\zero} |\leq\nud$ is unbounded and write 
$\ \big|\langle f_{\uom}\rangle_{\Box(\QQ)}\big|\leq \EE_{\ux} \nud (x) |D_{\QQ}(f_{\uom})\big|$. By the Cauchy-Schwarz inequality, we have 
\[
\big|\langle f_{\uom}\rangle_{\Box_{H,M}(\uQ_1,\ldots,\uQ_s)}\big|^2 \leq
\EE_{\ux}\,D_{\QQ}^2(f_{\uom})(x)+ 
\EE_{\ux}\,\big(\nud(\ux)-1\big)\,D_{\QQ}^2(f_{\uom})(x):= I + II.
\]
To estimate the first term note that 
\[
\EE_{\ux}\,D_{\QQ}^4(f_{\uom})(x)\leq
\EE_{\ux}\,D_{\QQ}^4(\nud)(x) \ll 1,
\]
by the polynomial forms condition and the assumption that the family $\QQ=(\uQ_1,\ldots,\uQ_s)$ is in general position. On the other hand 
\[
\EE_{\ux}\,\big|D_{\QQ}(f_{\uom})(x)\big|\leq \eps^C + o(1),
\]
as it is equivalent to the fact that
\[
\EE_{\ux}\,D_{\QQ}(f_{\uom})(x)\,g(x0 \leq \eps^C + o(1),
\]
uniformly for all functions $|g|\leq 1$ which follows from the inductive hypothesis on the number of unbounded factors. Then $I \leq \eps^C +o(1)$ by interpolation.\\

For the second term we may write again,
\begin{align*}
D_{\QQ}^2(f_{\uom})(x) &= \EE_{\uh,\uh'}\,
\EE_{\substack{\uy_1^{(0)},\uy_1^{(1)}\\
\uy_2^{(0)}\uy_2^{(1)}}}
\prd_{(\uom,\uom')\in\{0,1\}^{2s}}
f_{\uom}\big(\ux+\sum_{j=1}^s y_{1j}^{(\om_j)} \uQ_j(\uh)big)\,
f_{\uom}\big(\ux+\sum_{j=1}^s y_{2j}^{(\om_j)} \uQ_j(\uh')\big)\\
&= 
\EE_{\uh,\uh'}\,
\EE_{\substack{\uy_1^{(0)},\uy_1^{(1)}\\
\uy_2^{(0)}\uy_2^{(1)}}}
\prd_{(\uom,\uom')\in\{0,1\}^{2s}}
f_{(\uom,\uom')}\big(\ux+\sum_{j=1}^s y_{1j}^{(\om_j)} \uQ_j(\uh) + \sum_{j=1}^s y_{1j}^{(\om_j')} \uQ_j(\uh')
\big),
\end{align*}

\medskip

where $f_{(\uom,\uom')}=1$, unless $(\uom,\uom')=(\uom,\zero)$ or $(\uom,\uom')=(\zero,\uom)$ (with $\uom\neq \zero$), in which case $f_{(\uom,\uom')}=f_{\uom}$. Writing $\QQ'(\uh,\uh')=(\uQ_1(\uh),\ldots,\uQ_s(\uh),\uQ_1(\uh'),\ldots,\uQ_s(\uh'))$ for the extended polynomial system, we have that by box-norm Cauchy-Schwarz inequality

\begin{align*}
\big| \EE_{\ux}\,
\big(\nud (\ux)-1\big)\,D_{\QQ}^2(f_{\uom})(x)\big| &= 
\big|\langle \nud -1,\,D_{\QQ'}(f_{\uom,\uom'})\rangle\big|\\
& \leq \|\nud -1\|_{\Box_{H,M}(\QQ')}\,
\prod_{(\uom,{\uom}')\neq (\zero,\zero)}
\|f_{\uom,{\uom}'}\|_{\Box_{H,M}(\QQ')}=o(1).
\end{align*}
Indeed the extended system $\QQ'$ is also in general position (as $\pi_i(\uQ_j(\uh)\neq 0$ for $i\in [d],\ j\in [s]$), hence 
$\ \|\nud -1\|_{\Box_{H,M}(\QQ')}=o(1)$, while 
\[\|f_{\uom,{\uom}'}\|_{\Box_{H,M}(\QQ')}\leq
\|\nud \|_{\Box_{H,M}(\QQ')} +1 \ll 1,
\]
\medskip

by the polynomial forms condition. This proves the Proposition.
\end{proof}

\begin{cor}\label{cor4.2} Let $f:X\to\R$ satisfying $0\leq f\leq\nud$ and let $\eps>0$. Then there exists functions $0\leq g\leq 2$ and $\fh$, such that
\[
f=g+\fh,\quad\textit{and}\quad \|\fh\|_{\Box_{H,M}(\QQ)}
\leq \eps.
\]
\end{cor}

\begin{proof} Let $0\leq g\leq 2$ be given as Corollary \ref{cor4.1} with $\eps'=\eps^{1/C}$. Then by \eqref{4.13} and Proposition \ref{prop4.2} we have that 
\[
\|(f-g)\|_{\Box_{H,M}(\QQ)} \leq \eps+o(1).
\]
The Corollary follows by writing $\fh=f-g$.
\end{proof}

\section{Proof of Theorem 1.} 

The other main ingredient of obtaining polynomial patterns in relative dense subsets of the primes is using an appropriate version of the polynomial extension of Szemer\'{e}di's theorem due to Bergelson-Leibman \cite{BL96}. The version we use, given \cite{TZ08,TZ14}, is as follows.

\begin{thmE} Let $\de>0$, $\mP=(P_1,\ldots,P_l)$, $P_j\in\Z[y],\,P_j(0)=0$ be a polynomial map, and let 
$V=\{\vv_1,\ldots,\vv_l\}\subs\Z^d$.\\

If $g:X\to [0,1]$ is a function satisfying $\EE_{x\in X} g(x)\geq\de$, then one has 
\eq\label{5.1}
\La_{\mP,V}(g,\ldots,g)\geq c(\de)-o(1),
\ee
where $c(\de)>0$ is a constant, that depends only on $\de$, $V$ and polynomial map $\mP$.
\end{thmE}

\medskip

Now it is easy to prove our main result.

\begin{proof}[Proof of Theorem 1.] Given $A\subs \PP_{N'}^d$ with $|A|\geq \de\,|\PP_{N'}|^d$ we choose $W$, $N=[N/W]$,$\ub$ and the function $f_A$ as in \eqref{2.1}-{2.2}, then by \eqref{2.3}-\eqref{2.5} we have 

\[
\EE_{\ux\in X}\, f_A(\ux)\geq \de',\quad 0\leq f_A(\ux) \leq\nud (\ux),
\]
\medskip

with $\de'= c_0\de /4$. Then by Corollary \ref{cor4.2} with $\eps>0$ sufficiently small with respect to $c(\de')$, we have 
$f_A=g + \fh$ with $\|\fh\|_{\Box_{H,M}(\QQ)}
\leq \eps$, thus Corollary 3.1:

\[
\La_{\mP,V}(f_A,\ldots,f_A) \geq \La_{\mP,V}(g,\ldots,g)-
\eps^c-o(1) \geq c(\de')-\eps^c-o(1)\geq 
c(\de')/2-o(1)>0,
\]
\medskip

as long as $L\geq L(\mP,V)$ and $N\geq N(\mP,V,\de)$ is sufficiently large. 
\end{proof}

\section*{Acknowledgement} This research was performed while the
second author was a guest researcher at the HUN-REN Alfr\'{e}d R\'{e}nyi
Institute of Mathematics in the frame of the MTA Distinguished Guest
Fellow Scientist Program 2024 of the Hungarian Academy of Sciences, and was supported by the Hungarian National Research Developments and Innovation Office, NKFIH, Excellence 151341. 


\bigskip

{\small
\noindent
Andrew Lott\\
The University of Georgia\\
Athens, GA 30602, United States\\
e-mail: Andrew.Lott@uga.edu\\
\\
\'Akos Magyar\\
HUN-REN Alfr\'ed R\'enyi Institute of Mathematics \\
Budapest, Re\'altanoda u. 13--15, H-1053 Hungary\\
and\\
The University of Georgia\\
Athens, GA 30602, United States\\
e-mail: amagyar@uga.edu\\
\\
Giorgis Petridis\\
The University of Georgia\\
Athens, GA 30602, United States\\
email: petridis@uga.edu\\
\\
J\'anos Pintz\\
HUN-REN Alfr\'ed R\'enyi Institute of Mathematics \\
Budapest, Re\'altanoda u. 13--15, 
H-1053 Hungary\\
e-mail: pintz@renyi.hu}

\end{document}